\documentclass[12pt]{article}

\usepackage{inputenc}
\usepackage{amsmath}
\usepackage{amsfonts}
\usepackage{amssymb}
\usepackage{graphicx}
\usepackage[T1]{fontenc}
\usepackage[english]{babel}
\usepackage{mathrsfs}
\usepackage{amsthm}
\usepackage[all]{xy}
\usepackage{geometry}
\allowdisplaybreaks
\usepackage{color}

\title{A singular limit in a fractional reaction-diffusion equation with periodic coefficients}
\author{Alexis Léculier}
\begin{document}                     
\maketitle


\newtheorem{theorem}{Theorem}
\newtheorem{lemma}{Lemma}
\newtheorem{definition}{Definition}
\newtheorem{proposition}{Proposition}
\newtheorem{definition-proposition}{Definition-Proposition}
\newtheorem*{notation}{Notation}

\theoremstyle{definition}
\newtheorem{example}{Example}

\theoremstyle{definition}
\newtheorem*{remark}{Remark}

\begin{abstract}
We provide an asymptotic analysis of a non-local Fisher-KPP type equation in periodic media and with a non-local stable operator of order $\alpha \in (0,1)$. We perform a long time-long range scaling in order to prove that the stable state invades the unstable state with a speed which is exponential in time. 
\end{abstract}

\noindent{\bf AMS Class. No:} {35K57, 35B40, 35Q92.}

\bigskip

\section{Introduction}
\subsection{The equation}
We are interested in the following equation:
\begin{equation}
\label{equation_principale_1}
\left\{
\begin{aligned}
&\partial_t n(x,t) + L^{\alpha} (n) (x,t)= \mu (x) n(x,t) - n(x,t)^2, \ \ \   \ \ (x,t) \in \mathbb{R}^d \times[0,+\infty)\\
&n(x,0) = n_0(x) \in C^{\infty}_c(\mathbb{R}^d, \mathbb{R}^+).
\end{aligned}
\right.
\end{equation}
In the above setting, $\mu$ is a $1$-periodic function, $\alpha \in (0,1)$ is given and the term $L^\alpha$ denotes a fractional elliptic operator which is defined as follow:
\begin{equation}\label{def:L:alpha}
L^{\alpha}(n)(x,t) : = -PV \int_{\mathbb{R}^d} (n(x+h,t)-n(x,t))\beta(x,\frac{h}{|h|})\dfrac{dh}{|h|^{d+2\alpha}},
\end{equation}
where $\beta:\mathbb{R}^d \times S^{d-1} \rightarrow \mathbb{R}$ is a $1$-periodic smooth function such that for all $(x,\theta) \in \mathbb{R}^d \times S^{d-1}$
\[\beta(x,\theta)=\beta(x,-\theta) \ \text{ and } \ 0 < b \leq \beta(x,\theta) \leq B ,\]
with $b$ and $B$ positive constants. When $\beta$ is constant, we recover the classical fractional Laplacian $(-\Delta)^\alpha$. 

The main aim of this paper is to describe the propagation front associated to \eqref{equation_principale_1}. We show that the stable state invades the unstable state with an exponential speed. 

\subsection{The motivation}

Equation \eqref{equation_principale_1} models the growth and the invasion of a species subject to non-local dispersion in a heterogeneous environment. Such models describe the situations where individuals can jump (move rapidly) from one point to the other, for instance because of the wind for seeds or human transportation for animals. The function $n$ stands for the density of the population in position $x$ at time $t$. The diffusion term represented by the operator $L^{\alpha}$ describes the motions of individuals. The "logistic term" $\mu(x)n(x,t)-n(x,t)^2$ represents the growth rate of the population. The heterogeneity of the environment is modeled by the periodic function $\mu$. The regions where $\mu$ is positive represent areas where the species are favored whereas $\mu$ negative prevents the growth of the species. Conversely, the term $-n^2$ characterizes the death term because of some "logistic" considerations, as for example the quantity of food. 

The operator $L^\alpha$ generalizes the fractional Laplacian $(-\Delta)^\alpha$ which models "homogeneous" jumps : the individuals jump in every direction with the same frequency. Whereas the operator $L^\alpha$ models "heterogeneous" jumps : the individuals prefer to jump in the direction where $\beta$ is high. Also, the frequency of jumps will depend on the position $x$ of the individuals. Note that for the 1 dimensional case, for a regular bounded function $n$, $(1-\alpha) L^\alpha(n)(x)$ tends to $-\frac{\beta(x)}{4}n''(x)$ as $\alpha$ tends to $1^{-}$ which corresponds to a heterogeneous local diffusion. Moreover, the function $\beta$ will affect the principal eigenvalue $\lambda_1$ of $L^\alpha - \mu(x) Id$ (and consequently the negativity of $\lambda_1$ which is a criterion for the existence of a positive bounded stationary state). However, the hypothesis $0< b \leq \beta \leq B$ implies that the techniques used for the fractional Laplacian are robust and can be extended to the case of the operator $L^\alpha$.


Equation \eqref{equation_principale_1} was first introduced by Fisher in \cite{Fisher} (1937) and Kolmogorov, Petrovskii and Piscunov in \cite{KPP1937} (1938) in the particular case of a homogeneous environment ($\mu=1$) and a standard diffusion ($L^\alpha = -\Delta$) which corresponds to the case $\alpha =1$ and $\beta =1$. In \cite{Aronson-Weinberger}, Aronson and Weinberger proved a first similar result to our result for the case introduced by Fisher and Kolmogorov, Petrovskii and Piscunov. In this case, the propagation is with a constant speed independently of the direction. In \cite{Freidlin_Gertner}, Freidlin and Gärtner studied the question with a standard Laplacian in a heterogeneous environment ($\mu$ periodic). Using a probabilistic approach, they showed that the speed of the propagation is dependent on the chosen direction $e \in S^{d-1}$. But, the speed $c(e)$ in the direction $e$ is constant. Other proofs of this result, using PDE tools, can be found in \cite{Berestycki_Hamel_Nadirashvili} and \cite{Rossi_KPP}. In the case of the fractional Laplacian and a constant environment, Cabré and Roquejoffre in \cite{Cabre-Roquejoffre} proved the front position is exponential in time (see also for instance \cite{Del-Castillo-Negrete} for some heuristic and numerical works predicting such behavior and \cite{Mirrahimi1} for an alternative proof). Then in \cite{Cabre-Coulon-Roquejoffre}, Cabré, Coulon and Roquejoffre investigate the speed of propagation in a periodic environment modeled by equation \eqref{equation_principale_1} but considering the fractional Laplacian instead of the operator $L^\alpha$. One should underline the fact that in the fractional case, the speed of propagation does not depend anymore on the direction. They proved that the speed of propagation is exponential in time with a precise exponent depending on a periodic principal eigenvalue. 

The objective of this work is to provide an alternative proof of this property using an asymptotic approach known as "approximation of geometric optics". We will be interested in the longtime behavior of the solution $n$. We demonstrate that in the set $\left\{(x,t)\  |\  |x|<e^{\frac{|\lambda_1|t}{d+2\alpha}}\right\}$, as $t$ tends to infinity, $n$ converges to a stationary state $n_+$ and outside of this domain $n$ tends to zero. The main idea in this approach is to perform a long time-long range rescaling to catch the effective behavior of the solution (see for instance \cite{Opt_geom_2}, \cite{Opt_geom_1} and \cite{Opt_geom_3} for the classical Laplacian case). This paper is closely related to \cite{Mirrahimi1} where the authors Méléard and Mirrahimi have introduced such an "approximation" for a model with the fractional Laplacian and a simpler reaction term ($n-n^2$). A very recent work, \cite{Bouin_al}, uses also the techniques introduced in \cite{Mirrahimi1} (as known as the introduction of an adapting rescaling and the investigation of adapted sub and super solution) to investigate an integro-differential homogeneous Fisher-KPP type equation: the operator $L^\alpha$ is replaced by $J \ast n - n$ where the kernel $J$ is fat tailed but does not have singularity at the origin. 

This paper was initially written with a fractional Laplacian. At its completion, we became aware of a preprint by Souganidis and Tarfulea \cite{Souganidis_Tarfulea} which proves a result quite close to ours in the case of spatially periodic stable operators. Our proof is quite different since their approach is more based on the theory of viscosity solutions. We have verified that our approach works for the model treated in \cite{Souganidis_Tarfulea} with no additional idea. We present our result with the operator $L^\alpha$ given by \eqref{def:L:alpha}, where the proof for the fractional Laplacian applies almost word by word. In the course of the paper, we explain the points of our proof that allow to reach the generality of \cite{Souganidis_Tarfulea}.


\subsection{The assumptions}

For the initial data we will assume
\begin{equation} \label{H1} \tag{H1}
n_0 \in C^{\infty}_c(\mathbb{R}^d, \mathbb{R}^+), \ n_0 \not\equiv 0.
\end{equation}
The function $\mu$ is a $1$-periodic function, i.e.
\begin{equation}\label{H2} \tag{H2}
\forall k \in \left\lbrace1,...,d\right\rbrace, \ \mu(x_1,...,x_k+1,...,x_d)=\mu(x_1,...,x_d).
\end{equation}
Under the assumptions on $ \beta $, the operator $L^\alpha -\mu(x) Id$ admits a principal eigenpair $(\phi_1, \lambda_1)$ by the Krein-Rutman Theorem (see \cite{Roquejoffre1}) that is 
\begin{equation}
\label{vecteur_propre}
\left\{
\begin{aligned}
& L^{\alpha} \phi_1 (x)-\mu (x) \phi_1(x) =\lambda_1 \phi_1(x),\ x \in \mathbb{R}^d,\\
&\phi_1 \text{ periodic}, \ \phi_1>0, \ \left\|\phi_1 \right\|=1.
\end{aligned}
\right.
\end{equation} 
To assure the existence of a bounded, positive and periodic steady solution $n_+$ for \eqref{equation_principale_1}, we will assume that the principal eigenvalue $\lambda_1$ is negative :
\begin{equation}\label{H3} \tag{H3}
\lambda_1<0.
\end{equation}
Note that such stationary solution is unique in the class of positive and periodic stationary solutions (see \cite{Berestycki81}).

In section 4, we will study a more general equation : 
\begin{equation}
\label{equation_principale_KPP} 
\left\{
\begin{aligned}
&\partial_t n(x,t) + L^{\alpha} (n) (x,t)= F(x,n(x,t)), \ \ \   \ \ (x,t) \in \mathbb{R}^d \times[0,+\infty)\\
&n(x,0) = n_0(x) \in C^{\infty}_c(\mathbb{R}^d,\mathbb{R}^+).
\end{aligned}
\right.
\end{equation}
We make the following assumptions for $F$:
\begin{equation*}
\label{H4} \tag{H4}
\left\{
\begin{aligned}
&(i) \ \forall s \in \mathbb{R}, \ x \mapsto F(x,s) \text{ is periodic},\\
&(ii) \ F(x,0)=0,\\
&(iii) \  \exists \underline{c},\overline{C}>0 \ \text{ such that } \ \forall (x,s)\in \mathbb{R}^d\times \ \mathbb{R}, \ - \underline{c}\leq \partial_s(\dfrac{F(x,s)}{s})\leq-\overline{C}, \\
&(iv) \ \exists M>0, \forall(x,s)\in \mathbb{R}^d\times [M,+\infty[,\  F(x,s)<0.
\end{aligned}
\right.
\end{equation*}
We will denote $\partial_s (F)(x,0)$ by $\mu(x)$ and we still denote by $(\lambda_1, \phi_1)$ the principal eigenvalue and eigenfunction of $L^\alpha - \mu(x) Id$. We also still suppose \eqref{H3} (i.e. $\lambda_1$ is strictly negative).

\subsection{The main result and the method}

We introduce the following rescaling 
\begin{equation}
\label{rescaling}
(x,t) \longmapsto \left(|x|^{\frac{1}{\varepsilon}}\dfrac{x}{|x|}, \dfrac{t}{\varepsilon}\right).
\end{equation}
We perform this rescaling because one expects that for $t$ large enough, $n$ is close to the stationary state $n_+$ in the following set $\left\{(x,t)\in \mathbb{R}^d \times \mathbb{R}^+ \   |\  |x|<e^{\frac{|\lambda_1|t}{d+2\alpha}}\right\}$ and $n$ is close to $0$ in the set $\left\{(x,t)\in \mathbb{R}^d \times \mathbb{R}^+ \   |\  |x|>e^{\frac{|\lambda_1|t}{d+2\alpha}}\right\}$. The change of variable \eqref{rescaling} will therefore respect the geometries of these sets. We then rescale the solution of \eqref{equation_principale_1} as follows
\[n_{\varepsilon}(x,t)=n(|x|^{\frac{1}{\varepsilon}} \frac{x}{|x|}, \frac{t}{\varepsilon})\] 
and a new steady state
\[n_{+,\varepsilon}(x)=n_+ (|x|^{\frac{1}{\varepsilon}}\frac{x}{|x|}).\]
We prove:
\begin{theorem}\label{theorem_principal_2}
Assuming \eqref{H1}, \eqref{H2} and \eqref{H3}, let $n$ be the solution of \eqref{equation_principale_1}. Then\\
(i) $n_{\varepsilon} \rightarrow 0$, locally uniformly in $\mathcal{A} = \left\{(x,t) \in \mathbb{R}^d \times (0,\infty) | \ \ |\lambda_1 |\ t < (d+2\alpha) \log |x|\right\}$, \\
(ii) $\dfrac{n_{\varepsilon}}{n_{+,\varepsilon}} \rightarrow 1$, locally uniformly in $\mathcal{B} = \left\{(x,t) \in \mathbb{R}^d \times (0,\infty) | \  \   |\lambda_1 |\ t  > (d+2\alpha) \log |x|\right\}$.
\end{theorem}

To provide the main idea to prove Theorem \ref{theorem_principal_2}, we first explain the main element of the proof in the case of constant environment which was introduced in \cite{Mirrahimi1}.\\
A central argument to prove such result in the case of a constant environment, is that, using the rescaling \eqref{rescaling}, as $\varepsilon \rightarrow 0$, the term $\left((-\Delta)^\alpha (n)n^{-1}\right)(|x|^{\frac{1}{\varepsilon}-1}x,\frac{t}{\varepsilon})$ vanishes. More precisely, one can provide a sub and a super-solution to the rescaled equation which are indeed a sub and a super-solution to a perturbation of an ordinary differential equation derived from \eqref{equation_principale_1} by omitting the term $(-\Delta)^\alpha$. They also have the property that when one applies the operator $f \mapsto (-\Delta)^\alpha (f) f^{-1}$ to such functions, the outcome is very small and of order $\mathcal{O}(\varepsilon^2)$ as $\varepsilon$ tends to $0$.\\
In the case of periodic $\mu$, we use the same idea. However, in this case, the sub and super-solutions are multiplied by the principal eigenfunction and, the term $\left(L^\alpha (n) n^{-1} \right)(|x|^{\frac{1}{\varepsilon}-1}x,\frac{t}{\varepsilon})$ will not just tend to $0$ as in \cite{Mirrahimi1} but also compensate the periodic media. To prove the convergence of $n_\varepsilon$, dealing with this periodic term, we use the method of perturbed test functions from the theory of viscosity solutions and homogenization (introduced by Evans in \cite{Evans_visco_1} and \cite{Evans_visco_2}). Note that we also generalize the arguments of \cite{Mirrahimi1} to deal with a more general integral term $L^\alpha$ while in \cite{Mirrahimi1}, only the case of the fractional Laplacian was considered. In the last part, we will also generalize Theorem \ref{theorem_principal_2} to the case of Fisher-KPP reaction term:


\begin{theorem}
\label{theorem_principal_2_KPP}
Assuming \eqref{H1}, \eqref{H2}, \eqref{H3} and \eqref{H4}, let $n$ be the solution of \eqref{equation_principale_KPP}. Then
\\
(i) $n_{\varepsilon} \rightarrow 0$, locally uniformly in $\mathcal{A} = \left\{(x,t) \in \mathbb{R}^d \times (0,\infty) | \ \ |\lambda_1 |\ t < (d+2\alpha) \log |x|\right\}$, \\
(ii) $\dfrac{n_{\varepsilon}}{n_{+,\varepsilon}} \rightarrow 1$, locally uniformly in $\mathcal{B} = \left\{(x,t) \in \mathbb{R}^d \times (0,\infty) | \  \   |\lambda_1 |\ t  > (d+2\alpha) \log |x|\right\}$.
\end{theorem}
The proof of this Theorem follows from an adaptation of the proof of Theorem \ref{theorem_principal_2}.

In section 2, we introduce preliminary results and technical tools. In section 3, after  the rescaling, we provide a sub and a super-solution and demonstrate Theorem \ref{theorem_principal_2}. In section 4, we provide the points of the proof of Theorem \ref{theorem_principal_2_KPP} that differ from the proof of Theorem \ref{theorem_principal_2}.

\section{Preliminary results}

We first state a classical result on the fractional heat kernel.
\begin{proposition}\cite{Chen-Kumagai}
\label{heat_kernel}
There exists a positive constant $C$ larger than $1$ such that the heat kernel $p_\alpha(x,y,t)$ associated to the operator $\partial_t + L^\alpha$ verifies the following inequalities for $t>0$ :
\begin{equation}\label{heat_kernel_1}
C^{-1} \times \min(t^{-\frac{d}{2\alpha}}, \frac{t}{|x-y|^{d+2\alpha}})\leq p_\alpha (x,y,t) \leq C \times  \min(t^{-\frac{d}{2\alpha}}, \frac{t}{|x-y|^{d+2\alpha}}).
\end{equation}

\end{proposition}
The proof of this proposition is given in \cite{Chen-Kumagai}. \\
Now we use this proposition to demonstrate that beginning with a positive compactly supported initial data leads to a solution with algebraic tails. 

\begin{proposition}
Assuming \eqref{H1}, then there exists two constant $c_m$ and $c_M$ depending on $n_0$, $d$, $\alpha$ and $\mu$ such that :
 \[\dfrac{c_m}{1+|x|^{d+2\alpha}} \leq n(x,1) \leq \dfrac{c_M}{1+|x|^{d+2\alpha}}.\]
\end{proposition}

\begin{proof}
First, we define $M:=\max(\max n_0, \max |\mu|)$, we easily note that the constant functions $0$ and $M$ are respectively sub and super-solution to our problem. Then, thanks to the comparison principle (which is given in \cite{Cabre-Roquejoffre}), we have the following inequalities, for all $(x,t) \in \mathbb{R}^d \times [0,+\infty[$:
\[0 \leq n(x,t) \leq M.\]
Let $\underline{n}$ and $\overline{n}$ be the solutions of the two following systems : 
\begin{equation}
\left\{
\begin{aligned}
&\partial_t \underline{n} + L^\alpha \underline{n} =-2M \underline{n}, \\
&\underline{n} (x,0) = n_0(x),
\end{aligned}
\right.
\label{SUB}
\end{equation}
and 
 \begin{equation}
\left\{
\begin{aligned}
&\partial_t \overline{n} + L^\alpha \overline{n} = \max |\mu| \  \overline{n}, \\
&\overline{n} (x,0) = n_0(x).
\end{aligned}
\right.
\label{SUP}
\end{equation}
Thanks to Proposition \ref{heat_kernel}, we can solve \eqref{SUB} and find
\[\underline{n}(x,t)=e^{-2Mt} \int_{\mathbb{R}^d}p_\alpha(x,y,t)  n_0(y)dy,\]
Thus for any $t>0$, we obtain
\begin{align*}
& e^{-2Mt}\int_{supp(n_0)} C^{-1}\times n_0(y) \min(t^{-\frac{d}{2\alpha}}, \frac{t}{|x-y|^{d+2\alpha}})dy \leq \underline{n}(x,t)\\
&\Rightarrow  e^{-2M}\int_{supp(n_0)} C^{-1}\times n_0(y) \min(1, \frac{1}{|x-y|^{d+2\alpha}})dy \leq \underline{n}(x,1).
\end{align*} 
Thanks to the dominated convergence theorem, we have:
\[(1+|x|^{d+2\alpha})\times  e^{-2M}\int_{supp(n_0)} C^{-1}\times n_0(y) \min(1, \frac{1}{|x-y|^{d+2\alpha}})dy\underset{|x|\rightarrow \infty}{\longrightarrow}  e^{-2M}\int_{supp(n_0)} C^{-1} \times n_0(y) dy.\]
Therefore, we conclude by a compactness argument that for any $x \in \mathbb{R}^d$:
\begin{equation} 
\frac{e^{-2M}C^{-1}}{(1+|x|^{d+2\alpha})}\leq \underline{n}(x,1),
\end{equation}
where the last C is a new constant depending only on $n_0$. Moreover thanks to the comparison principle, we have that for any $t \geq 0$
\begin{equation}\label{sub_solution}
\underline{n}(x,t) \leq n(x,t) \Rightarrow \frac{e^{-2M}C^{-1}}{(1+|x|^{d+2\alpha})}\leq n(x,1).
\end{equation}
In the same way, we can solve \eqref{SUP} and the solution is 
\[\overline{n}(x,t) = e^{\max |\mu| t}\int_{\mathbb{R}^d} n_0(y)\times p_\alpha (t,x,y)dy.\]
Using similar arguments, we get that for all $x \in \mathbb{R}^d$, 
\begin{equation}
\label{super_solution}
n(x,1) \leq \overline{n}(x,1) \leq \dfrac{C e^{\max |\mu| t}}{(1+|x|^{d+2\alpha})}.
\end{equation}
By combining \eqref{sub_solution} and \eqref{super_solution} together, we finally obtain 
\begin{equation} \label{inegalité_H1'}
\dfrac{c_m}{1+|x|^{d+2\alpha}} \leq n(x,1) \leq \dfrac{c_M}{1+|x|^{d+2\alpha}}.
\end{equation}
\end{proof} 

We next provide a technical lemma which will be useful all along the article. The main ideas of the proof of the lemma come from \cite{Mirrahimi1} by S. Méléard and S. Mirrahimi for Point $(i)$ and \cite{Coulon} by A.C. Coulon Chalmin for Point $(ii)$. To this end, we first introduce the computation of $L^\alpha$ of a product of functions:
\[L^{\alpha} (fg) (x,t) = f(x,t) L^{\alpha} g(x,t) +  g(x,t) L^{\alpha} f(x,t) - \widetilde{K}(f,g) (x,t), \]
with,  
\[\widetilde{K}(f,g) (x,t):= C' \ PV \int_{\mathbb{R}^d}\frac{(f(x,t)-f(x+h,t))(g(x,t)-g(x+h,t))}{|h|^{d+2\alpha}}\beta(x,\frac{h}{|h|})dy.\]

\begin{lemma}
\label{lemme_utile}
Let $\gamma$ be a positive constant such that
\[\gamma \in  \left\{
    \begin{array}{ll}
        [0, 2\alpha[ & \mbox{if } \alpha<\frac{1}{2} \\
        ]2\alpha-1,1[ & \mbox{if } \frac{1}{2}\leq \alpha < 1,
    \end{array}
		\right.
\]
$\chi : \mathbb{R} \rightarrow \mathbb{R}^d$ be a $C^1$ periodic, strictly positive function and $g(x):=\frac{1}{1+|x|^{d+2\alpha}}$. Then there exists a positive constant $C$, which does not depend on $x$, such that, for all $x \in \mathbb{R}^d$: \\
(i) for all $a>0$,
\[| L^{\alpha} g (ax) | \leq a^{2 \alpha} C g(ax), \]
(ii) for all $a \in ]0,1]$, 
\[|\widetilde{K}(g(a .),\chi)(x) | \leq \frac{C a ^{2 \alpha - \gamma}}{1+(a |x|)^{d+2\alpha}} = C a^{2 \alpha - \gamma} g(a x).\]
\end{lemma}
The proof is given in Appendix A. Note that we will not use the assumption $b\leq \beta$ in the proof of Lemma \ref{lemme_utile} (but only $\beta \leq B$). The assumption $b \leq \beta$ is necessary to Proposition \ref{heat_kernel} and also to ensure the existence and the positiveness of $\phi_1$.  

\begin{remark}
If we want to reach the same level of generality as in \cite{Souganidis_Tarfulea}, we just have to adapt the previous Lemma to an operator $L^\alpha$ with a kernel $\beta$ of the form $\beta(x,y)$ where $\beta$ is a $1$-periodic with respect to $x$, smooth function from $ \mathbb{R}^d \times \mathbb{R}^d$ such that for all $(x,y) \in \mathbb{R}^d \times \mathbb{R}^{d}$
\[\beta(x,y)=\beta(x,-y) \ \text{ and } \ 0 < b \leq \beta(x,y) \leq B ,\]
with $b$ and $B$ positive constants. The interested reader can verify that the proof of
Lemma 1 is robust enough and can easily be adapted to such kernels. 
\end{remark}

\section{The proof of Theorem \ref{theorem_principal_2}}
In this section we will provide the proof of Theorem \ref{theorem_principal_2}. Let us rewrite \eqref{equation_principale_1} with respect to the rescaling given by \eqref{rescaling}
\begin{equation}
\label{equation intermediaire}
\varepsilon \partial_t (\ n( |x|^{\frac{1}{\varepsilon}}\frac{x}{|x|},\frac{t}{\varepsilon})\ ) = -L^{\alpha} (n) ( |x|^{\frac{1}{\varepsilon}} \frac{x}{|x|},\frac{t}{\varepsilon}) + n(|x|^{\frac{1}{\varepsilon}}\frac{x}{|x|},\frac{t}{\varepsilon})[\mu (|x|^{\frac{1}{\varepsilon}} \frac{x}{|x|}) - n( |x|^{\frac{1}{\varepsilon}} \frac{x}{|x|},\frac{t}{\varepsilon})]. 
\end{equation}
\begin{notation} 
For any function $v:\mathbb{R}^d \times \mathbb{R}^+ \rightarrow \mathbb{R}$ and $w: \mathbb{R}^d \rightarrow \mathbb{R}$ we denote by $v_\varepsilon$ and $w_\varepsilon$ the rescaled functions given by :
\[v_{\varepsilon}(x,t):=v( |x|^{\frac{1}{\varepsilon}} \frac{x}{|x|},\frac{t}{\varepsilon})\ \text{ and } \ w_\varepsilon (x)= w (|x|^{\frac{1}{\varepsilon}-1}x).\]
\end{notation}
One can write the first term in the right hand side of \eqref{equation intermediaire} in term of $n_\varepsilon$ in the following way. 
\begin{align*}
&L^{\alpha} (n) (|x|^{\frac{1}{\varepsilon}-1} x,\frac{t}{\varepsilon}) = - PV \int_{\mathbb{R}^d} \frac{n ( |x|^{\frac{1}{\varepsilon}-1} x + h,\frac{t}{\varepsilon} )-n(|x|^{\frac{1}{\varepsilon}-1}x,\frac{t}{\varepsilon})}{|h|^{2\alpha +d}}\times \beta(|x|^{\frac{1}{\varepsilon}-1}x,\frac{h}{|h|})dh \\
&= - PV \int_{\mathbb{R}^d}\left( n_{\varepsilon} \left( \left||x|^{\frac{1}{\varepsilon}-1} x + h\right|^{\varepsilon}  \frac{(|x|^{\frac{1}{\varepsilon}-1} x + h)}{ ||x|^{\frac{1}{\varepsilon}-1} x + h|},t\right)-n_{\varepsilon}(x,t)\right)\frac{\beta_\varepsilon(x,\frac{h}{|h|})dh}{|h|^{2\alpha +d}}.
\end{align*}
We can hence define:
\[L^{\alpha}_{\varepsilon}(n_{\varepsilon})(x,t):=  L^{\alpha} (n) ( |x|^{\frac{1}{\varepsilon}} \frac{x}{|x|},\frac{t}{\varepsilon}),\]
which allows us to write \eqref{equation intermediaire} as below: 
\begin{equation}
\label{equation_principale_2}
\varepsilon \partial_t  n_{\varepsilon}( x,t) = -L^{\alpha}_{\varepsilon} n_{\varepsilon} (x,t) + n_{\varepsilon}(x,t)[\mu_{\varepsilon} (x) - n_{\varepsilon}(x,t)]. 
\end{equation}
In the same way we define 
\[\widetilde{K}_\varepsilon(n_\varepsilon, \chi_\varepsilon)(x,t) := \widetilde{K}(n,\chi)(|x|^{\frac{1}{\varepsilon}-1}x, \frac{t}{\varepsilon}).\]
Moreover, according to the inequalities \eqref{inegalité_H1'}, we can consider $n(x,1)$ as our initial data instead of $n(x,0)$. So we can replace the assumption \eqref{H1} by:
\begin{equation}
\label{H1'}\tag{H1'}
\dfrac{c_m}{1+|x|^{d+2\alpha}} \leq n_0(x) \leq \dfrac{c_M}{1+|x|^{d+2\alpha}} \Rightarrow \dfrac{c_m}{1+|x|^{\frac{d+2\alpha}{\varepsilon}}} \leq n_{0,\varepsilon}(x) \leq \dfrac{c_M}{1+|x|^{\frac{d+2\alpha}{\varepsilon}}}.
\end{equation}

In the next subsection we are going to provide sub and super-solutions to \eqref{equation_principale_2} which will allow us to demonstrate Theorem \ref{theorem_principal_2} in a second subsection. 
\subsection{Sub and super-solution to (\ref{equation_principale_2}).}

\begin{theorem}
\label{sub_super_solution_theorem}
We assume \eqref{H2} and \eqref{H3} and we choose positive constants $C_m <\frac{|\lambda_1|}{\max \phi_1}$ and $C_M>\frac{|\lambda_1|}{\min \phi_1}$ and $\delta$ such that
\[0<\delta \leq \min(\sqrt{C_M \min \phi_1 - |\lambda_1|},\sqrt{  |\lambda_1|-C_m \max \phi_1}).\]
Then there exists a positive constant $\varepsilon_0<\delta$ such that for all $\varepsilon \in]0,\varepsilon_0[$ we have: \\
(i) $ f^{M}_{\varepsilon}(t,x)=\phi_{1,\varepsilon}(x) \times \frac{C_M}{1+e^{-\frac{t}{\varepsilon}(|\lambda_1|+\varepsilon^2)-\frac{\delta}{\varepsilon}}|x|^{\frac{d+2\alpha}{\varepsilon}}} \ \text{ is a super-solution of \eqref{equation_principale_2}}$,\\
(ii) $ f^m_{\varepsilon}(x,t)=\phi_{1,\varepsilon}(x) \times \frac{C_m e^{-\frac{\delta}{\varepsilon}}}{1+e^{-\frac{t}{\varepsilon}(|\lambda_1|-\varepsilon^2)-\frac{\delta}{\varepsilon}}|x|^{\frac{d+2\alpha}{\varepsilon}}} \ \text{ is a sub-solution of \eqref{equation_principale_2}}.$\\
(iii) Moreover, if we assume \eqref{H1'} and $C_m<\dfrac{c_m}{\max |\phi_1|}$ and $C_M>\dfrac{c_M}{\min |\phi_1|}$ where $c_m$ and $c_M$ are given by \eqref{H1'} then for all $(x,t) \in \mathbb{R}^d \times [0,+\infty[$,
\begin{equation}
\phi_{1,\varepsilon}(x) \times \frac{C_m e^{\frac{-\delta}{\varepsilon}-\varepsilon t}}{1+e^{-\frac{|\lambda_1|t+\delta}{\varepsilon}}|x|^{\frac{d+2\alpha}{\varepsilon}}}\leq n_{\varepsilon}(x,t) \leq \phi_{1,\varepsilon}(x) \times \frac{C_M e^{\varepsilon t}}{1+e^{-\frac{|\lambda_1|t+\delta}{\varepsilon}}|x|^{\frac{d+2\alpha}{\varepsilon}}}.
\label{inequation_2}
\end{equation}
\end{theorem}

\begin{proof}
Since the proofs of $(i)$ and $(ii)$ follow from similar arguments, we will only provide the proof of $(i)$ and $(iii)$.\\
\textit{\textbf{Proof of (i).}} We define:
\begin{equation}\label{definition_psi}
\psi(x,t):= \frac{C_M}{1+e^{-t(|\lambda_1|+\varepsilon^2)-\frac{\delta}{\varepsilon}}|x|^{d+2\alpha}}.
\end{equation}
Then, noticing that $\phi_1$ is independent of $t$, we first bound $\partial_t \psi_{\varepsilon}$ from below,
\begin{equation}
\label{derivee_temporelle}
\begin{aligned}
\partial_t  \psi_{\varepsilon} (x,t) &= \frac{C_M\frac{(|\lambda_1|+\varepsilon^2)}{\varepsilon}e^{-t\frac{(|\lambda_1|+\varepsilon^2)}{\varepsilon}-\frac{\delta}{\varepsilon}}|x|^{\frac{d+2\alpha}{\varepsilon}}}{(1+e^{-t\frac{(|\lambda_1|+\varepsilon^2)}{\varepsilon}-\frac{\delta}{\varepsilon}}|x|^{\frac{d+2\alpha}{\varepsilon}})^2} \\
&= \frac{  \psi_{\varepsilon}(x,t) }{\varepsilon} [(|\lambda_1|+\varepsilon^2) \frac{e^{-t\frac{(|\lambda_1|+\varepsilon^2)}{\varepsilon}-\frac{\delta}{\varepsilon}}|x|^{\frac{d+2\alpha}{\varepsilon}}}{1+e^{-t\frac{|\lambda_1|+\varepsilon^2)}{\varepsilon}-\frac{\delta}{\varepsilon}}|x|^{\frac{d+2\alpha}{\varepsilon}}} ] \\
&\geq  \frac{  \psi_{\varepsilon}(x,t) }{\varepsilon} [|\lambda_1|+\varepsilon^2-  \psi_{\varepsilon}(x,t) \phi_{1,\varepsilon}(x) ].
\end{aligned}
\end{equation}
The last inequality is obtained from the definition of $C_M$ and $\varepsilon$. Actually, for such $C_M$ and $\varepsilon$, we have, for any positive non-null constant A, the following relation: 
\[\frac{A(|\lambda_1|+\varepsilon^2)}{1+A} \geq |\lambda_1|+\varepsilon^2 - \frac{C_M \min \phi_1}{1+A},\]
because,
\[
\begin{aligned}
 |\lambda_1|+\varepsilon^2 - \frac{C_M\min \phi_1}{1+A} & = \frac{(1+A)( |\lambda_1|+\varepsilon^2) - C_M\min \phi_1}{1+A} \\
																							& = \frac{A( |\lambda_1|+\varepsilon^2) - (C_M \min \phi_1 -|\lambda_1|-\varepsilon^2)}{1+A} \\
																							& \leq \frac{A( |\lambda_1|+\varepsilon^2) }{1+A}.
\end{aligned}
\]
We also compute $L^{\alpha}_{\varepsilon} (f^{M}_{\varepsilon}) (x,t)$ as a fractional Laplacian of a product of functions,  
\begin{equation}
\label{equation_produit}
L^{\alpha}_{\varepsilon} (f^{M}_{\varepsilon}) (x,t) =  \phi_{1,\varepsilon}(x) L^{\alpha}_{\varepsilon} \psi_{\varepsilon} (x,t) +  \psi_{\varepsilon}(x,t) L^{\alpha}_{\varepsilon} \phi_{1,\varepsilon} (x) - \widetilde{K}_{\varepsilon}(\psi_\varepsilon, \phi_{1,\varepsilon}) (x,t)
\end{equation}
with $\widetilde{K}$ given in section 2. Replacing this in equation \eqref{equation_principale_2} and using the two previous results \eqref{derivee_temporelle} and \eqref{equation_produit}, we find: 
\begin{equation*}
\begin{aligned}
&\varepsilon \partial_t f^M_{\varepsilon}(x,t) +L^{\alpha}_{\varepsilon} f^M_{\varepsilon} (x,t) -f^M_{\varepsilon}(x,t)[\mu_{\varepsilon} (x) - f^M_{\varepsilon}(x,t)] \\
&\geq f^M_{\varepsilon}(x,t) (|\lambda_1| + \varepsilon^2 - f^M_{\varepsilon}(x,t))+\phi_{1,\varepsilon}(x)L^{\alpha}_{\varepsilon} \psi_{\varepsilon} (x,t) +  \psi_{\varepsilon}(x,t)L^{\alpha}_{\varepsilon} \phi_{1,\varepsilon} (x) \\
&- \widetilde{K}_{\varepsilon}(\psi_\varepsilon, \phi_{1,\varepsilon})(x,t)- \mu_{\varepsilon}(x)  f^M_{\varepsilon}(x,t) +  f^M_{\varepsilon}(x,t)^2\\
&= \varepsilon^2 f^M_{\varepsilon}(x,t) + \phi_{1,\varepsilon}(x) L^{\alpha}_{\varepsilon} \psi_{\varepsilon} (x,t) - \widetilde{K}_{\varepsilon}(\psi_\varepsilon, \phi_{1,\varepsilon}) (x,t),
\end{aligned}
\end{equation*}
where we have used \eqref{vecteur_propre} and \eqref{H3} for the last equality.\\
In order to control $ L^{\alpha}_{\varepsilon} \psi_{\varepsilon} (x,t)$ and $\widetilde{K}_{\varepsilon}(\psi_\varepsilon, \phi_{1,\varepsilon})(x,t)$, we are going to use Lemma \ref{lemme_utile}. For, $L^{\alpha}_{\varepsilon} \psi_{\varepsilon} (x,t)$, noticing that $\psi_\varepsilon(x,t)= C_M g(e^{\frac{-t(|\lambda_1|+\varepsilon^2)-\delta}{\varepsilon(1+2\alpha)}}|x|^{\frac{1}{\varepsilon}-1}x)$, and thanks to the point $(i)$ of Lemma \ref{lemme_utile}, we obtain: 
\[- C  e^{-2\alpha \frac{t(|\lambda_1|+\varepsilon^2)+\delta}{\varepsilon}}\psi_{\varepsilon}(t,x) \leq L^{\alpha}_{\varepsilon}\psi_{\varepsilon}(t,x) .\]
But, comparing the growths, there exists $\varepsilon_1>0$ such that for $\varepsilon <\varepsilon_1$ and for all $t\geq 0$ :
\[C_M \times C e^{-2\alpha\frac{t(|\lambda_1|+\varepsilon^2)+\delta}{\varepsilon(d+2\alpha)}}-\frac{\varepsilon^2}{3} \leq 0,\]
hence: 
\[- \frac{\varepsilon^2}{3}\psi_{\varepsilon}(x,t ) \leq L^{\alpha}_{\varepsilon}\psi_\varepsilon (x,t) .\]
Now we deal with $\widetilde{K}_{\varepsilon}(\psi_\varepsilon, \phi_{1,\varepsilon}) (x,t)$ in a similar fashion. Thanks to Lemma \ref{lemme_utile} $(ii)$, we find: 
\begin{align*}
\widetilde{K}_{\varepsilon}(\psi_\varepsilon, \phi_{1,\varepsilon})(x,t) &=\widetilde{K}(\psi,\phi_1)(|x|^{\frac{1}{\varepsilon}}\frac{x}{|x|},\frac{t}{\varepsilon})\\
&  \leq  C  e^{- \frac{(2\alpha-\gamma)[t(|\lambda_1|+\varepsilon^2)+\delta]}{\varepsilon}}\psi(|x|^{\frac{1}{\varepsilon}}\frac{x}{|x|},\frac{t}{\varepsilon}) \\
&= C  e^{- \frac{(2\alpha-\gamma)[t(|\lambda_1|+\varepsilon^2)+\delta]}{\varepsilon}}\psi_{\varepsilon}(x,t).
\end{align*}
Then, noticing  that for any choice of $\alpha$, $2\alpha - \gamma$ is strictly positive, we deduce there exists $\varepsilon_2>0$ such that for all $\varepsilon < \varepsilon_2$:
\[C e^{-(2\alpha-\gamma) \frac{t(|\lambda_1|+\varepsilon^2)+\delta}{\varepsilon}}-\frac{\varepsilon^2 \min{\phi_1} }{3} \leq 0.\]
We deduce that 
\[\widetilde{K}_{\varepsilon}(\psi_\varepsilon, \phi_{1,\varepsilon})(x,t)  \leq  \frac{\varepsilon^2}{3}\psi_{\varepsilon}(x,t) \min \phi_1 \leq  \frac{\varepsilon^2}{3}\psi_{\varepsilon}(x,t)\phi_{1,\varepsilon}(x).\]
We set :
\[\varepsilon_0= \min (\varepsilon_1, \varepsilon_2).\]
Then,  we conclude that for $\varepsilon \leq \varepsilon_0$, we have : 
\begin{equation*}
\begin{aligned}
&\varepsilon \partial_t  f^M_{\varepsilon}(x,t)+L^{\alpha}_{\varepsilon} f^M_{\varepsilon} (x,t) -f^M_{\varepsilon}(x,t)[\mu_{\varepsilon} (x) - f^M_{\varepsilon}(x,t)] \\
&\geq \varepsilon^2 f^M_{\varepsilon}(x,t) + \phi_{1,\varepsilon}(x) L^{\alpha}_{\varepsilon} \psi_{\varepsilon} (x,t) - \widetilde{K}_{\varepsilon}(\psi,\phi_1) (x,t)\\
 &\geq \varepsilon^2 f^M_{\varepsilon}(x,t) -\frac{\varepsilon^2}{3} \phi_{1,\varepsilon}( x)\psi_{\varepsilon} (x,t) - \frac{\varepsilon^2}{3} \phi_{1,\varepsilon}(x) \psi_{\varepsilon}(x,t) \\
& \geq \frac{\varepsilon^2}{3} f^M_{\varepsilon}(x,t)\\
& \geq 0.
\end{aligned}
\end{equation*}
Therefore $f^M_{\varepsilon}$ is a super-solution of \eqref{equation_principale_2} and this concludes the proof of the point $(i)$.
\bigbreak
\textit{\textbf{Proof of (iii).}} From \eqref{H1'}, since $\max | \phi_1|C_m<c_m$ and $c_M<C_M \min|\phi_1|$, we have:
\[f_\varepsilon^m(x,0)= \dfrac{\phi_{1,\varepsilon}(x) \times C_m e^{-\frac{\delta}{\varepsilon}}}{1+e^{-\frac{\delta}{\varepsilon}}|x|^{\frac{d+2\alpha}{\varepsilon}}} = \dfrac{\phi_{1,\varepsilon}(x) \times C_m }{e^{\frac{\delta}{\varepsilon}}+|x|^{\frac{d+2\alpha}{\varepsilon}}} \leq \dfrac{c_m }{1+|x|^{\frac{d+2\alpha}{\varepsilon}}} \leq n_\varepsilon(x,0) \leq f_\varepsilon^M(x,t).\] 
Then, according to the comparison principle,  we obtain:
\[\phi_{1,\varepsilon}(x) \times \frac{C_m e^{-\frac{\delta}{\varepsilon}}}{1+e^{-\frac{t}{\varepsilon}(|\lambda_1|-\varepsilon^2)-\frac{\delta}{\varepsilon}}|x|^{\frac{d+2\alpha}{\varepsilon}}} \leq n_{\varepsilon}(x,t) \leq  \phi_{1,\varepsilon}(x) \times \frac{C_M}{1+e^{-\frac{t}{\varepsilon}(|\lambda_1|+\varepsilon^2)-\frac{\delta}{\varepsilon}}|x|^{\frac{d+2\alpha}{\varepsilon}}},\]
and hence
\begin{equation}
\phi_{1,\varepsilon}(x) \times \frac{C_m e^{\frac{-\delta}{\varepsilon}-\varepsilon t}}{1+e^{-\frac{|\lambda_1|t+\delta}{\varepsilon}}|x|^{\frac{d+2\alpha}{\varepsilon}}}\leq n_{\varepsilon}(x,t) \leq \phi_{1,\varepsilon}(x) \times \frac{C_M e^{\varepsilon t}}{1+e^{-\frac{|\lambda_1|t+\delta}{\varepsilon}}|x|^{\frac{d+2\alpha}{\varepsilon}}}.
\tag{\ref{inequation_2}}
\end{equation}

\end{proof}

\subsection{Convergence to the stationary state}

Thanks to the inequalities \eqref{inequation_2}, we can now prove Theorem \ref{theorem_principal_2}. To prove this theorem, we are going to follow the ideas of Méléard and Mirrahimi in \cite{Mirrahimi1}. 

\begin{proof}
First, we perform a Hopf-Cole transformation 
\begin{equation}
\label{HC} 
u_{\varepsilon} (x,t) : = \varepsilon \log n_{\varepsilon}(x,t) \ \text{ and } \ u_{+,\varepsilon} (x) : = \varepsilon \log n_{+,\varepsilon}(x).
\end{equation}
Taking the logarithm in \eqref{inequation_2} and multiplying by $\varepsilon$, we find :  
\begin{align*}
&-\varepsilon^2 t + \varepsilon \log C_m \phi_{1,\varepsilon} - \varepsilon \log (1+e^{-\frac{|\lambda_1|t+\delta}{\varepsilon}}|x|^{\frac{d+2\alpha}{\varepsilon}})-\delta\leq u_{\varepsilon} (x,t)\\
 & \text{ and } u_\varepsilon(x,t) \leq \varepsilon^2 t + \varepsilon \log C_M \phi_{1,\varepsilon} - \varepsilon \log (1+e^{-\frac{|\lambda_1|t+\delta}{\varepsilon}}|x|^{\frac{d+2\alpha}{\varepsilon}}).
\end{align*}
Define 
\[\underline{u}(x,t)=\underset{\varepsilon \rightarrow 0}{\liminf}\  u_{\varepsilon}(x,t), \ \ \overline{u}(x,t)=\underset{\varepsilon \rightarrow 0}{\limsup}\  u_{\varepsilon}(x,t), \ \ \ \text{ for all } (x,t) \in \mathbb{R}^d\times(0,+\infty).\]
Letting $\varepsilon \rightarrow 0$, we obtain 
\[\min(0,|\lambda_1|\ t + \delta  - (d+2\alpha)\log |x|)-\delta \leq \underline{u}(x,t) \leq \overline{u}(x,t) \leq \min(0,|\lambda_1|\ t + \delta - (d+2\alpha)\log |x|).\]
We then let $\delta \rightarrow 0$ and we obtain 
\[u(x,t):=\underline{u}(x,t)=\overline{u}(x,t) =  \min(0,|\lambda_1|\ t - (d+2\alpha)\log |x|).\]
We deduce that $u_\varepsilon$ converges locally uniformly in $\mathbb{R}^d\times [0,+\infty[$ to $u$ since the above limits are locally uniform in $\varepsilon$. 

\textit{\textbf{Proof of (i).}} For any compact set $K$ in $\mathcal{A}$, there exists a positive constant $a$ such that for all $(x_0,t_0) \in K$, we have $u(x_0,t_0) < -a$. It is thus immediate from \eqref{HC} that $n_{\varepsilon}$ converges uniformly to $0$ in $K\subset\mathcal{A}$. This concludes the proof of $(i)$.
\bigbreak
\textit{\textbf{Proof of (ii).}} We divide \eqref{equation_principale_2} by $n_\varepsilon$ and we obtain
\[ \partial_t u_\varepsilon + L^{\alpha}_\varepsilon n_\varepsilon n_\varepsilon^{-1} = \mu_\varepsilon -n_\varepsilon, \]
that we rewrite as below, 
\begin{equation}\label{nouvelle_equation_n+}
n_\varepsilon=n_{+,\varepsilon} +(- \partial_t u_\varepsilon- L^{\alpha}_\varepsilon n_\varepsilon n_\varepsilon^{-1} + \mu_\varepsilon -n_{+,\varepsilon}).
\end{equation}
\bigbreak
\textbf{Step 1: }\textit{$\dfrac{n_\varepsilon (x_0,t_0)}{n_{+,\varepsilon}(x_0)} \geq 1 +o(1)$ in every compact set of $\mathcal{B}$.}

Let $K$ be a compact set of $\mathcal{B}$ and $(x_0,t_0) \in K$. We choose $\nu$ a positive constant small enough such that for all $(y,s) \in K$, 
\begin{equation}\label{condition_K}
(d+2\alpha) \log|y| < |\lambda_1|s-2\nu \ \text{ and } \ \ 2\nu<|\lambda_1| s.
\end{equation}
First, we define 
\[\underline{\varphi}(x,t):=\min (0, -(d+2\alpha)\log|x|+|\lambda_1|t_0-\nu)-(t-t_0)^2.\]
It is easy to verify that $u-\underline{\varphi}$ achieves a local strict in $t$ and a global in $x$ minimum at $(x_0,t_0)$. 
Then, we define 
\[\underline{\varphi}_\varepsilon(x,t):=-\varepsilon \log (1+e^{-\frac{|\lambda_1|t_0-\nu}{\varepsilon}}|x|^{\frac{d+2\alpha}{\varepsilon}})-(t-t_0)^2.\]
Thus, $(\underline{\varphi}_{\varepsilon})_{\varepsilon}$ converges locally uniformly to $\underline{\varphi}$. Moreover, since $n_{+}$ is periodic and strictly positive, we have that $u_{+,\varepsilon}$ converges to $0$, hence $u_{\varepsilon}-(\underline{\varphi}_{\varepsilon}+ u_{+,\varepsilon}) \underset{\varepsilon \rightarrow 0}{\longrightarrow} u-\underline{\varphi}$ locally uniformly.
Thus, there exists $(x_{\varepsilon},t_{\varepsilon}) \in \mathbb{R}^d \times [0,+\infty[$ such that $(x_{\varepsilon},t_\varepsilon)$ is a minimum point (local in $t$ and global in $x$) of $(u_{\varepsilon}-\underline{\varphi}_\varepsilon -  u_{+,\varepsilon})$ and $(u_{\varepsilon}-\underline{\varphi}_\varepsilon -  u_{+,\varepsilon})(x_\varepsilon,t_\varepsilon) \rightarrow 0$. Since $(x_0,t_0)$ is a strict in $t$ local minimum of $u-\underline{\varphi}$, one can choose $t_\varepsilon$ such that $t_\varepsilon \rightarrow t_0$. We deduce that
\begin{equation}  \label{derivee_en_temps_o}
\partial_t u_\varepsilon(x_\varepsilon, t_\varepsilon) = \partial_t \underline{\varphi}_\varepsilon (x_\varepsilon,t_\varepsilon) = -2(t_\varepsilon-t_0)=o(1).
\end{equation}

One should ensure that $(x_\varepsilon)_{\varepsilon \rightarrow 0}$ have all their accumulation points in $\overline{B}(0,e^{\frac{|\lambda_1|t_0-\nu}{d+2\alpha}})$ as $\varepsilon$ tends to $0$. This is the case because, at time $t=t_0$, in $\overline{B}(0, e^{\frac{|\lambda_1|t_0-\nu}{d+2\alpha}})$, $u_\varepsilon - \underline{\varphi}_\varepsilon - u_{+,\varepsilon}$ tends to $0$, whereas in $\overline{B}(0, e^{\frac{|\lambda_1|t_0-\nu}{d+2\alpha}})^c$, $u_\varepsilon - \underline{\varphi}_\varepsilon - u_{+,\varepsilon}$ tends to a strictly positive function. \\
We deduce that there exists $\varepsilon_1>0$ such that for all $\varepsilon<\varepsilon_1$ we have $x_\varepsilon\in \overline{B}(0,e^{\frac{|\lambda_1|t_0-\frac{\nu}{2}}{d+2\alpha}})$. 
\\
Then we continue by proving $(- L^{\alpha}_\varepsilon (n_\varepsilon )n_\varepsilon^{-1} + \mu_\varepsilon -n_{+,\varepsilon})(x_\varepsilon,t_\varepsilon)\geq o(1)$, 
\[- L^{\alpha}_\varepsilon (n_\varepsilon) n_\varepsilon^{-1}(x_\varepsilon, t_\varepsilon) =\int_{\mathbb{R}^d} (e^{\frac{u_\varepsilon \left( \left||x_\varepsilon|^{\frac{1}{\varepsilon}-1}x_\varepsilon+h \right|^{\varepsilon-1}(|x_\varepsilon|^{\frac{1}{\varepsilon}-1}x_\varepsilon+h), t_\varepsilon \right)-u_\varepsilon(x_\varepsilon,t_\varepsilon)}{\varepsilon}}-1) \frac{\beta_\varepsilon(x,\frac{h}{|h|})dh}{|h|^{d+2\alpha}}.
\]

From the definition of $(x_\varepsilon,t_\varepsilon)$, we have for all $y \in \mathbb{R}^d$ :
\[ (u_\varepsilon-\underline{\varphi}_\varepsilon- u_{+,\varepsilon})(x_\varepsilon,t_\varepsilon)\leq (u_\varepsilon-\underline{\varphi}_\varepsilon- u_{+,\varepsilon})(y,t_\varepsilon) ,\]
and thus 
\[ (\underline{\varphi}_\varepsilon+ u_{+,\varepsilon})(y,t_\varepsilon)-(\underline{\varphi}_\varepsilon+ u_{+,\varepsilon})(x_\varepsilon,t_\varepsilon) \leq u_\varepsilon(y,t_\varepsilon)-u_\varepsilon(x_\varepsilon,t_\varepsilon).\]
Therefore, from \eqref{HC} we have
\[ -L^{\alpha}_{\varepsilon}(e^{\frac{\underline{\varphi}_\varepsilon}{\varepsilon}}n_{+,\varepsilon})(e^{\frac{\underline{\varphi}_\varepsilon}{\varepsilon}}n_{+,\varepsilon})^{-1}(x_\varepsilon,t_\varepsilon)\leq -L^{\alpha}_\varepsilon (n_\varepsilon )n_\varepsilon^{-1}(x_\varepsilon, t_\varepsilon).\]
Finally, using that $n_{+,\varepsilon}$ is solution of the stationary equation, we obtain 
\[\begin{aligned}
(- L^{\alpha}_\varepsilon (n_\varepsilon) n_\varepsilon^{-1} + \mu_\varepsilon -n_{+,\varepsilon})(x_\varepsilon,t_\varepsilon) &\geq (-L^{\alpha}_{\varepsilon}(e^{\frac{\underline{\varphi}_\varepsilon}{\varepsilon}}n_{+,\varepsilon}))(e^{\frac{\underline{\varphi}_\varepsilon}{\varepsilon}}n_{+,\varepsilon})^{-1}+\mu_\varepsilon -n_{+,\varepsilon})(x_\varepsilon,t_\varepsilon)\\
&=  (-L^{\alpha}_{\varepsilon}(e^{\frac{\underline{\varphi}_\varepsilon}{\varepsilon}})(e^{-\frac{\underline{\varphi}_\varepsilon}{\varepsilon}})-L^{\alpha}_{\varepsilon}(n_{+,\varepsilon})(n_{+,\varepsilon})^{-1}\\
&+\widetilde{K}_\varepsilon (e^{\frac{\underline{\varphi}_\varepsilon}{\varepsilon}},n_{+,\varepsilon})(e^{\frac{\underline{\varphi}_\varepsilon}{\varepsilon}}n_{+,\varepsilon})^{-1}+\mu_\varepsilon -n_{+,\varepsilon} )(x_\varepsilon,t_\varepsilon)\\
&= \left(-L^{\alpha}_{\varepsilon}(e^{\frac{\underline{\varphi}_\varepsilon}{\varepsilon}})(e^{-\frac{\underline{\varphi}_\varepsilon}{\varepsilon}})+\widetilde{K}_\varepsilon (e^{\frac{\underline{\varphi}_\varepsilon}{\varepsilon}},n_{+,\varepsilon})(e^{\frac{\underline{\varphi}_\varepsilon}{\varepsilon}}n_{+,\varepsilon})^{-1}\right)(x_\varepsilon,t_\varepsilon).
\end{aligned}
\]

In order to control the last two terms of the above inequality, we are going to use Lemma \ref{lemme_utile}. Note that, we have the following link between $e^{\frac{\underline{\varphi}_\varepsilon}{\varepsilon}}$ and $g(x)=\frac{1}{1+|x|^{d+2\alpha}}$:
\[e^{\frac{\underline{\varphi}_\varepsilon}{\varepsilon}}(x,t)=\frac{e^{-\frac{(t-t_0)^2}{\varepsilon}}}{1+e^{-\frac{|\lambda_1|t_0-\nu}{\varepsilon}}|x|^{\frac{d+2\alpha}{\varepsilon}}} = e^{\frac{-(t-t_0)^2}{\varepsilon}} g (e^{-\frac{|\lambda_1|t_0-\nu}{(d+2\alpha)\varepsilon}} |x|^{\frac{1}{\varepsilon}-1} x).\]
And so, we can deduce from Lemma \ref{lemme_utile} that: 
\[ o(1)=-C e^{-\frac{2\alpha(|\lambda_1|t_0-\nu)}{(d+2\alpha)\varepsilon}} \leq -L^{\alpha}_{\varepsilon}(e^{\frac{\underline{\varphi}_\varepsilon}{\varepsilon}})(e^{-\frac{\underline{\varphi}_\varepsilon}{\varepsilon}})(x_\varepsilon,t_\varepsilon) ,\]
and, 
\[o(1)=-C \: e^{-\frac{(2\alpha-\gamma)(|\lambda_1|t_0-\nu)}{\varepsilon}} n_{+,\varepsilon}(x_\varepsilon, t_\varepsilon)^{-1} \leq \widetilde{K}_{\varepsilon}(e^{\frac{\underline{\varphi}_\varepsilon}{\varepsilon}},n_{+,\varepsilon})(e^{\frac{\underline{\varphi}_\varepsilon}{\varepsilon}}n_{+,\varepsilon})^{-1}(x_\varepsilon,t_\varepsilon) .\]
We deduce that:
\begin{equation}
o(1) \leq (- L^{\alpha}_\varepsilon (n_\varepsilon) n_\varepsilon^{-1} + \mu_\varepsilon -n_{+,\varepsilon})(x_\varepsilon,t_\varepsilon).
\end{equation}
\\
Finally, combining the above inequality with \eqref{nouvelle_equation_n+} and \eqref{derivee_en_temps_o}, we obtain that
\[1+ o(1) \leq \dfrac{n_\varepsilon(x_\varepsilon, t_\varepsilon)}{n_{+,\varepsilon}( x_\varepsilon)} .\]
Now, we want to bring back this inequality at the point $(x_0,t_0)$. There are two cases: 
\bigbreak
\textbf{Case 1: } $|x_\varepsilon|\geq|x_0|$\\
Because of the definition of $\underline{\varphi}_\varepsilon$, we have :
\[\underline{\varphi}_{\varepsilon}(x_{\varepsilon},t_{\varepsilon}) \leq \underline{\varphi}_\varepsilon (x_0,t_0).\]
Since $(x_\varepsilon, t_\varepsilon)$ is a minimum point of $u_\varepsilon-(\underline{\varphi}_\varepsilon+ u_{\varepsilon,+})$, we deduce that
\[ u_\varepsilon (x_\varepsilon,t_\varepsilon) -  u_{+,\varepsilon} (x_\varepsilon) \leq u_\varepsilon (x_0,t_0) -  u_{+,\varepsilon}(x_0).\]
Thanks to \eqref{HC}, it follows that  
\[1+o(1) \leq \dfrac{n_{\varepsilon}(x_\varepsilon,t_\varepsilon)}{n_{+,\varepsilon}(x_\varepsilon)}  \leq \dfrac{n_{\varepsilon}(x_0,t_0)}{n_{+,\varepsilon}(x_0)} .\]

\bigbreak
\textbf{Case 2: }$ |x_\varepsilon|< |x_0|$\\
In this case, since $(x_0,t_0) \in K$ and thanks to \eqref{condition_K}, we have that
\[-|\lambda_1|t_0+\nu+(d+2\alpha)\log(|x_\varepsilon|) \leq -|\lambda_1|t_0+\nu+(d+2\alpha)\log(|x_0|)\leq -\nu <0,\]
and thus $e^{\frac{-|\lambda_1|t_0+\nu+(d+2\alpha)\log(|x_\varepsilon|) }{\varepsilon}} = o(1)$. We deduce that
\[e^{-\frac{\underline{\varphi}_{\varepsilon}(x_\varepsilon,t_\varepsilon)}{\varepsilon}}= e^{\frac{(t_\varepsilon-t_0)^2}{\varepsilon}}(1+e^{\frac{-|\lambda_1|t_0+\nu+(d+2\alpha)\log(|x_\varepsilon|)}{\varepsilon}}) \geq 1.\] 
Moreover, following similar computations, we obtain that $ e^{\frac{\underline{\varphi}_{\varepsilon}(x_0,t_0)}{\varepsilon}} =1+o(1).$
Hence, from the definition of $(x_\varepsilon, t_\varepsilon)$, we get
\[u_{\varepsilon}(x_\varepsilon,t_\varepsilon) -  u_{+,\varepsilon}(x_\varepsilon)+\underline{\varphi}_{\varepsilon}(x_0,t_0)-\underline{\varphi}_{\varepsilon}(x_\varepsilon,t_\varepsilon) \leq u_{\varepsilon}(x_0,t_0)- u_{+,\varepsilon}(x_0) .\]
Thanks to \eqref{HC}, we obtain that 
\[1+o(1) \leq \frac{n_{\varepsilon}(x_\varepsilon,t_\varepsilon)}{n_{+,\varepsilon}(x_\varepsilon)} \times \dfrac{e^{\frac{\underline{\varphi}_{\varepsilon}(x_0,t_0)}{\varepsilon}}}{e^{\frac{\underline{\varphi}_{\varepsilon}(x_\varepsilon,t_\varepsilon)}{\varepsilon}}}  \leq \dfrac{n_{\varepsilon}(x_0,t_0)}{n_{+,\varepsilon}(x_0)}.\]
So we have proved, in all cases

\[ 1 +o(1) \leq \dfrac{n_\varepsilon (x_0,t_0)}{n_{+,\varepsilon}(x_0)}.\]

\bigbreak
\textbf{Step 2: }\textit{$\dfrac{n_\varepsilon (x_0,t_0)}{n_{+,\varepsilon}(x_0)} \leq 1 +o(1)$ in every compact set of $\mathcal{B}$.}

This step is very similar to the first one. \\
We pick $(x_0,t_0) \in \mathcal{B}$ and let $\nu$ be a positive constant. As before, we define 
\[\overline{\varphi}(x,t):=\min (0,|\lambda_1|t_0+\nu-(d+2\alpha)\log|x|)+(t-t_0)^2.\]

It is easy to verify that $u-\overline{\varphi}$ achieves a local and strict in $t$ and a global in $x$ maximum at $(x_0,t_0)$. 
Then, defining 
\[\overline{\varphi}_\varepsilon(x,t):=-\varepsilon \log (1+e^{-\frac{|\lambda_1|t_0+\nu}{\varepsilon}} |x|^{\frac{d+2\alpha}{\varepsilon}})+(t-t_0)^2 ,\]
we have that $(\overline{\varphi}_{\varepsilon})_{\varepsilon}$ converges locally uniformly to $\overline{\varphi}$. Moreover,we know that $u_{+,\varepsilon}$ tends to $0$ and so $u_{\varepsilon}-(\overline{\varphi}_{\varepsilon}+ u_{+,\varepsilon}) \underset{\varepsilon \rightarrow 0}{\longrightarrow} 0$ uniformly in $\mathcal{B}$.
Thus, there exists $(x_{\varepsilon},t_{\varepsilon}) \in \mathbb{R}^d \times [0,+\infty[$ such that $(x_{\varepsilon},t_\varepsilon)$ is a maximum point, global in $x$ and local in $t$, of $(u_{\varepsilon}-\overline{\varphi}_\varepsilon - \varepsilon u_{+,\varepsilon})$ and 
\begin{equation}\label{futur_contradiction}
(u_{\varepsilon}-\overline{\varphi}_\varepsilon -  u_{+,\varepsilon})(x_\varepsilon,t_\varepsilon) \rightarrow 0.
\end{equation}
Since $(x_0,t_0)$ is a strict in $t$ local maximum of $u-\overline{\varphi}$, one can choose $t_\varepsilon$ such that $t_\varepsilon \rightarrow t_0$. We deduce that
\begin{equation}\label{dérivée_temporelle_step2}
\partial_t u_\varepsilon(x_\varepsilon, t_\varepsilon) = \partial_t \overline{\varphi}_\varepsilon (x_\varepsilon,t_\varepsilon) = 2(t_\varepsilon-t_0)=o(1).
\end{equation}

One should ensure that $(x_\varepsilon)_{\varepsilon \rightarrow 0}$ have all their accumulation points in $\overline{B}(0, e^{\frac{|\lambda_1|t_0+\frac{\nu}{4}}{d+2\alpha}})$. This is the case because for $\varepsilon$ small enough, in $\overline{B}(0, e^{\frac{|\lambda_1|t_0}{d+2\alpha}})$, $u_\varepsilon-\overline{\varphi}_\varepsilon-u_{+,\varepsilon}$ tends to 0 whereas in $\overline{B}(0, e^{\frac{|\lambda_1|t_0+\frac{\nu}{4}}{d+2\alpha}})^c$, $u_\varepsilon-\overline{\varphi}_\varepsilon-u_{+,\varepsilon}$ is lower than a strictly negative function.\\
We deduce that there exists $\varepsilon_2>0$ such that for all $\varepsilon<\varepsilon_2$ we have 
\begin{equation}\label{localisation xvarepsilon}
x_\varepsilon \in \overline{B}(0,e^{\frac{|\lambda_1|t_0+\frac{\nu}{2}}{d+2\alpha}}).
\end{equation} 

Then we continue by showing $(- L^{\alpha}_\varepsilon (n_\varepsilon) n_\varepsilon^{-1} + \mu_\varepsilon -n_{+,\varepsilon})\leq o(1)$. With similar computations as in step 1, we obtain :
\[(- L^{\alpha}_\varepsilon (n_\varepsilon) n_\varepsilon^{-1} + \mu_\varepsilon -n_{+,\varepsilon})(x_\varepsilon,t_\varepsilon) \leq (-L^{\alpha}_{\varepsilon}(e^{\frac{\overline{\varphi}_\varepsilon}{\varepsilon}})(e^{-\frac{\overline{\varphi}_\varepsilon}{\varepsilon}})+\ \widetilde{K}_\varepsilon (e^{\frac{\overline{\varphi}_\varepsilon}{\varepsilon}},n_{+,\varepsilon})(e^{\frac{\overline{\varphi}_\varepsilon}{\varepsilon}}n_{+,\varepsilon})^{-1})(x_\varepsilon,t_\varepsilon).\]
Since
\[e^{\frac{\overline{\varphi}_\varepsilon}{\varepsilon}}(x,t)=\frac{e^{\frac{(t-t_0)^2}{\varepsilon}}}{1+e^{-\frac{|\lambda_1|t_0+\nu}{\varepsilon}}|x|^{\frac{d+2\alpha}{\varepsilon}}} = e^{\frac{(t-t_0)^2}{\varepsilon}} g (e^{-\frac{|\lambda_1|t_0+\nu}{\varepsilon(d+2\alpha)}} |x|^{\frac{1}{\varepsilon}-1}x),\]
we can deduce thanks to Lemma \ref{lemme_utile} that : 
\[(-L^{\alpha}_{\varepsilon}(e^{\frac{\overline{\varphi}_\varepsilon}{\varepsilon}})(e^{-\frac{\overline{\varphi}_\varepsilon}{\varepsilon}})(x_\varepsilon,t_\varepsilon) \leq C e^{-\frac{2\alpha(|\lambda_1|t_0+\nu)}{(d+2\alpha)\varepsilon}}=o(1),\]
and, 
\[\ \widetilde{K}_{\varepsilon}(e^{\frac{\overline{\varphi}_\varepsilon}{\varepsilon}},n_{+,\varepsilon})(e^{\frac{\overline{\varphi}_\varepsilon}{\varepsilon}}n_{+,\varepsilon})^{-1}( x_\varepsilon,t_\varepsilon) \leq C \: e^{-\frac{(2\alpha-\gamma)(|\lambda_1|t_0+\nu)}{\varepsilon(d+2\alpha)}} n_{+,\varepsilon}(x_\varepsilon, t_\varepsilon)^{-1} = o(1).\]
Finally, combining the two previous inequalities and \eqref{dérivée_temporelle_step2} in \eqref{nouvelle_equation_n+} we have obtained
\[\dfrac{n_\varepsilon(x_\varepsilon, t_\varepsilon)}{n_{+,\varepsilon}( x_\varepsilon)}\leq 1+o(1).\]

Then, there are two cases to bring it back to the point $(x_0,t_0)$:

\textbf{Case 1: } $|x_\varepsilon|\leq |x_0|$
By definition of $\overline{\varphi}_{\varepsilon}$, we have:
\[ \overline{\varphi}_\varepsilon (x_0,t_0)\leq \overline{\varphi}_{\varepsilon}(x_{\varepsilon},t_{\varepsilon}).\]
Since $(x_\varepsilon, t_\varepsilon)$ is a maximum point of $u_\varepsilon-(\overline{\varphi}_\varepsilon+ u_{\varepsilon,+})$, we deduce that
\[ u_\varepsilon (x_0,t_0) -  u_{+,\varepsilon}(x_0) \leq u_\varepsilon (x_\varepsilon,t_\varepsilon) -  u_{+,\varepsilon} (x_\varepsilon).\]
Thanks to \eqref{HC}, it follows that,
\[\dfrac{n_{\varepsilon}(x_0,t_0)}{n_{+,\varepsilon}(x_0)} \leq \dfrac{n_{\varepsilon}(x_\varepsilon,t_\varepsilon)}{n_{+,\varepsilon}(x_\varepsilon)} \leq 1+o(1).\]
\bigbreak
\textbf{Case 2: } $|x_\varepsilon|>|x_0|$
Thanks to  \eqref{localisation xvarepsilon}, there exists $\varepsilon_2$ such that for all positive $\varepsilon<\varepsilon_2$ there holds
\[|x_{\varepsilon}| \leq e^{\frac{|\lambda_1|t_0+\frac{\nu}{2}}{d+2\alpha}} \Rightarrow -|\lambda_1|t_0-\nu+(d+2\alpha)\log|x_\varepsilon|<-\dfrac{\nu}{2}.\]
And thus, 
\[e^{-\frac{\overline{\varphi}_{\varepsilon}(x_\varepsilon,t_\varepsilon)}{\varepsilon}}= e^{\frac{-(t_\varepsilon-t_0)^2}{\varepsilon}}(1+e^{\frac{-|\lambda_1|t_0-\nu+(d+2\alpha)\log(|x_\varepsilon|)}{\varepsilon}}) \leq 1+e^{\frac{-\nu}{2\varepsilon}} \leq 1+o(1).\]
Moreover, we know by definition that $ e^{\frac{\overline{\varphi}_{\varepsilon}(x_0,t_0)}{\varepsilon}} =1+o(1).$ \\
Furthermore, by definition of $(x_\varepsilon,t_\varepsilon)$, we have 
\[u_{\varepsilon}(x_0,t_0)- u_{+,\varepsilon}(x_0) \leq u_{\varepsilon}(x_\varepsilon,t_\varepsilon) -  u_{+,\varepsilon}(x_\varepsilon)+\overline{\varphi}_{\varepsilon}(x_0,t_0)-\overline{\varphi}_{\varepsilon}(x_\varepsilon,t_\varepsilon),\]
Combining the above inequalities and thanks to \eqref{HC} we obtain that 
\[\dfrac{n_{\varepsilon}(x_0,t_0)}{n_{+,\varepsilon}(x_0)} \leq \frac{n_{\varepsilon}(x_\varepsilon,t_\varepsilon)}{n_{+,\varepsilon}(x_\varepsilon)} \times \dfrac{e^{\frac{\overline{\varphi}_{\varepsilon}(x_0,t_0)}{\varepsilon}}}{e^{\frac{\overline{\varphi}_{\varepsilon}(x_\varepsilon,t_\varepsilon)}{\varepsilon}}} \leq 1+o(1).\]
So we have proved, in all cases
\[\dfrac{n_\varepsilon (x_0,t_0)}{n_{+,\varepsilon}(x_0)} \leq 1 +o(1).\]
Passing up to the limit, we finally obtain the result of \textit{(ii)}. 
\end{proof}

\section{Generalization to KPP type reaction terms}

We can generalize our result to a model with a reaction term $F(x,s)$ which verifies the Fisher KPP assumptions given by \eqref{H4}.

\begin{example}
Obviously we can take as before 
\[F(x,s)=\mu(x)s-s^2.\]
\end{example}

\begin{example}
We can generalize it to the classical example:
\[F(x,s)=\mu(x)s-\omega(x) s^2.\]
Where $\mu$ is a continuous periodic function, and $\omega$ is a continuous periodic strictly positive function.
\end{example}

Of course, we keep the main idea of the previous proof: the rescaling \eqref{rescaling}. So the equation \eqref{equation_principale_KPP} becomes:
\begin{equation}
\label{equation_principale_2_KPP}
\varepsilon \partial_t  n_{\varepsilon}( x,t) = -L^{\alpha}_{\varepsilon} n_{\varepsilon} (x,t) + F_\varepsilon(x,n_\varepsilon(x,t)). 
\end{equation}
As before, according to the comparison principle, the point \textit{(ii)} and \textit{(iii)} of \eqref{H4} imply
\begin{equation}\label{nouvelle_assertion_F}
\mu(x) n - \underline{c} n^2 \leq F(x,n) \leq \mu(x) n - \overline{C} n^2.
\end{equation}
If we associate this result to \eqref{heat_kernel_1}, one can still obtain that the solution will have algebraic tails at time $t=1$ and hence one can replace the assumption \eqref{H1} by \eqref{H1'}:
\begin{equation}\tag{H1'}
\dfrac{c_m}{1+|x|^{\frac{d+2\alpha}{\varepsilon}}} \leq n_{0,\varepsilon}(x) \leq \dfrac{c_M}{1+|x|^{\frac{d+2\alpha}{\varepsilon}}}.
\end{equation}
Therefore, we still have the same sub and super-solutions:
\begin{theorem}\label{theorem_sous_sur_solution_2}
We assume \eqref{H2}, \eqref{H3} and \eqref{H4} and if we choose $C_m <\frac{|\lambda_1|}{\underline{c}\max \phi_1}$ and $C_M> \frac{|\lambda_1|}{\overline{C}\min \phi_1}$ where $\underline{c}$ and $\overline{C}$ are given by the assumptions $(iii)$ of \eqref{H4} and a positive constant $\delta$ such that
\[0<\delta \leq \min(\sqrt{\overline{C}C_M \min \phi_1 - |\lambda_1|},\sqrt{  |\lambda_1|-\underline{c}C_m \max \phi_1});\]
then there exists a positive constant $\varepsilon_0<\delta$ such that for all $\varepsilon \in]0,\varepsilon_0[$ we have:\\
(i) $ f^{M}_{\varepsilon}(x,t)=\phi_{1,\varepsilon}(x) \times \frac{C_M}{1+e^{-\frac{t}{\varepsilon}(|\lambda_1|+\varepsilon^2)-\frac{\delta}{\varepsilon}}|x|^{\frac{d+2\alpha}{\varepsilon}}} \ \text{ is a super-solution of \eqref{equation_principale_2_KPP}},$\\
(ii) $ f^m_{\varepsilon}(x,t)=\phi_{1,\varepsilon}(x) \times \frac{C_me^{-\frac{\delta}{\varepsilon}}}{1+e^{-\frac{t}{\varepsilon}(|\lambda_1|-\varepsilon^2)-\frac{\delta}{\varepsilon}}|x|^{\frac{d+2\alpha}{\varepsilon}}} \ \text{ is a sub-solution of \eqref{equation_principale_2_KPP}}$.\\
(iii) Moreover, if we assume \eqref{H1'} and $C_m<\dfrac{c_m}{\max |\phi_1|}$ and $C_M>\dfrac{c_M}{\min |\phi_1|}$ with $c_m$ and $c_M$ given by \eqref{H1'} then for all $(x,t) \in \mathbb{R}^d \times [0,+\infty[$,
\begin{equation}
\phi_{1,\varepsilon}(x) \times \frac{C_m e^{\frac{-\delta}{\varepsilon}-\varepsilon t}}{1+e^{-\frac{|\lambda_1|t+\delta}{\varepsilon}}|x|^{\frac{d+2\alpha}{\varepsilon}}}\leq n_{\varepsilon}(x,t) \leq \phi_{1,\varepsilon}(x) \times \frac{C_M e^{\varepsilon t}}{1+e^{-\frac{|\lambda_1|t+\delta}{\varepsilon}}|x|^{\frac{d+2\alpha}{\varepsilon}}}.
\label{inequation_2_KPP}
\end{equation}
\end{theorem}

\begin{proof}
Here is the main step of the proof of the point $(i)$. As in the proof of Theorem \ref{sub_super_solution_theorem}, we put:
\[f^{M}_{\varepsilon}(t,x)=\phi_{1,\varepsilon}(x) \times \psi_\varepsilon (x,t),\]
where $\psi_\varepsilon$ is given by \eqref{definition_psi}, but with a constant $C_M$ given in Theorem \ref{theorem_sous_sur_solution_2}.
Then, with similar computations as before, we find:
\[\partial_t f^{M}_{\varepsilon}\geq \frac{f^{M}_{\varepsilon}}{\varepsilon}(|\lambda_1|+\varepsilon^2-\overline{C}f^{M}_{\varepsilon}).\]
Therefore, using \eqref{nouvelle_assertion_F} and Lemma \ref{lemme_utile}, we get:
\begin{align*}
&\varepsilon \partial_t f^M_{\varepsilon}(x,t) +L^{\alpha}_{\varepsilon} f^M_{\varepsilon} (x,t) -F_\varepsilon(x,f_\varepsilon(x,t)) \\
 &\geq f^M_{\varepsilon}(x,t) (|\lambda_1| + \varepsilon^2 - \overline{C}f^M_{\varepsilon}(x,t))+\phi_{1,\varepsilon}(x) L^{\alpha}_{\varepsilon} \psi_{\varepsilon} (x,t) +  \psi_{\varepsilon}(x,t) L^{\alpha}_{\varepsilon} \phi_{1,\varepsilon} (x) \\
&- \widetilde{K}_{\varepsilon}(\psi, \phi_1) (x,t)- \mu_{\varepsilon}(x)  f^M_{\varepsilon}(x,t) + \overline{C} f^M_{\varepsilon}(x,t)^2\\
&\geq \varepsilon^2 f^M_{\varepsilon}(x,t) -\frac{\varepsilon^2}{3}  f^M_{\varepsilon}(x,t)- \frac{\varepsilon^2}{3}  f^M_{\varepsilon}(x,t) + \psi_{\varepsilon}(x,t) \left[L^{\alpha}_{\varepsilon} \phi_{1,\varepsilon} (x) - (\lambda_1 + \mu_{\varepsilon}(x))\phi_{1,\varepsilon} (x)\right]\\
&\geq 0.
\end{align*}
Thus, we have demonstrated the point $(i)$. The proof of the point $(ii)$ follows similar arguments. We do not give the proof of the point $(iii)$ because this is similar to the proof of $(iii)$ of Theorem \ref{sub_super_solution_theorem}: the main argument is the comparison principle.
\end{proof}

Thus, we can perform the Hopf-Cole transformation \eqref{HC} and we obtain that $u_\varepsilon$ converges locally uniformly to:
\[  u(x,t)=\min(0,|\lambda_1|\ t - (d+2\alpha)\log |x|).\]
Therefore the part \textit{(i)} of Theorem \ref{theorem_principal_2_KPP} can be proved following similar arguments as in the proof of \textit{(i)} of Theorem \ref{theorem_principal_2}. The proof of $(ii)$ changes a little bit so we are going to provide the demonstration.

\begin{proof}[\textbf{Proof of (ii) of Theorem \ref{theorem_principal_2_KPP}.}] Dividing by $n_\varepsilon$ in \eqref{equation_principale_2}, we obtain
\begin{equation} \label{equation_divise_KPP}
 \partial_t u_\varepsilon + L^{\alpha}_\varepsilon n_\varepsilon n_\varepsilon^{-1} = \dfrac{F_\varepsilon(x,n_\varepsilon)}{n_\varepsilon}.
\end{equation}
\bigbreak
\textbf{Step 1: }\textit{$\dfrac{n_\varepsilon (x_0,t_0)}{n_{+,\varepsilon}(x_0)} \geq 1 +o(1)$ in every compact set of $\mathcal{B}$.}\\
The main difference with the proof of Theorem \ref{theorem_principal_2} is that from \eqref{equation_divise_KPP}, we do not obtain directly $\dfrac{n_\varepsilon(x_\varepsilon,t_\varepsilon)}{n_{+,\varepsilon}(x_\varepsilon) }\geq 1+o(1)$ but we deduce $\dfrac{F_\varepsilon(x_\varepsilon,n_{+,\varepsilon})}{n_{+,\varepsilon}}(x_\varepsilon,t_\varepsilon)-\dfrac{F_\varepsilon(x_\varepsilon,n_\varepsilon)}{n_\varepsilon}(x_\varepsilon,t_\varepsilon)\geq o(1)$.

Let $K$ be a compact set of $\mathcal{B}$ and $(x_0,t_0) \in K$. We choose $\nu$ a positive constant small enough such that for all $(y,s) \in K$, 
\begin{equation}\label{condition_K_KPP}
(d+2\alpha) \log|y| < |\lambda_1|s-2\nu \ \text{ and } \ \ 2\nu<|\lambda_1|s.
\end{equation}
First, we define 
\[\underline{\varphi}(t,x):=\min (0, -(d+2\alpha)\log|x|+|\lambda_1|t_0-\nu)-(t-t_0)^2.\]
It is easy to verify that $u-\underline{\varphi}$ achieves a local strict in $t$ and a global in $x$ minimum at $(x_0,t_0)$. 
Then, we define 
\[\underline{\varphi}_\varepsilon(x,t):=-\varepsilon \log (1+e^{-\frac{|\lambda_1|t_0-\nu}{\varepsilon}}|x|^{\frac{d+2\alpha}{\varepsilon}})-(t-t_0)^2.\]
Thus, $(\underline{\varphi}_{\varepsilon})_{\varepsilon}$ converges locally uniformly to $\underline{\varphi}$. We know that $u_{+,\varepsilon}$ tends to $0$ and so $u_{\varepsilon}-(\underline{\varphi}_{\varepsilon}+ u_{+,\varepsilon}) \underset{\varepsilon \rightarrow 0}{\longrightarrow} u-\underline{\varphi}$ locally uniformly. Thus, there exists $(x_{\varepsilon},t_{\varepsilon}) \in \mathcal{B}$ such that $(x_{\varepsilon},t_\varepsilon)$ is a minimum point of $(u_{\varepsilon}-\underline{\varphi}_\varepsilon -  u_{+,\varepsilon})$ and $(u_{\varepsilon}-\underline{\varphi}_\varepsilon - u_{+,\varepsilon})(x_\varepsilon,t_\varepsilon) \rightarrow 0$. Since $(x_0,t_0)$ is a strict local minimum of $u-\underline{\varphi}$ in $t$, we can choose $t_\varepsilon$ such that $t_\varepsilon \rightarrow t_0$. Then 
\begin{equation}\label{deérivée_temporelle_KPP_Step1}
\partial_t u_\varepsilon(x_\varepsilon, t_\varepsilon) = \partial_t \underline{\varphi}_\varepsilon (x_\varepsilon,t_\varepsilon) = -2(t_\varepsilon-t_0)=o(1).
\end{equation}
With the same reasoning as in the proof of Theorem \ref{theorem_principal_2}, we get that there exists $\varepsilon_1>0$ such that for $\varepsilon<\varepsilon_1$, $x_\varepsilon \in \overline{B}(0, e^{\frac{|\lambda_1|t_0-\frac{\nu}{2}}{d+2\alpha}})$. 

Then we continue by proving 
\[(- L^{\alpha}_\varepsilon (n_\varepsilon) n_\varepsilon^{-1} + \dfrac{F_\varepsilon(x,n_\varepsilon)}{n_\varepsilon})(x_\varepsilon,t_\varepsilon)\geq (\dfrac{F_\varepsilon(x,n_\varepsilon)}{n_\varepsilon}-\dfrac{F(x_\varepsilon, n_{+,\varepsilon})}{n_{+,\varepsilon}})(x_\varepsilon, t_\varepsilon)+ o(1).\]
We know that 
\[- L^{\alpha}_\varepsilon (n_\varepsilon) n_\varepsilon^{-1}(x_\varepsilon, t_\varepsilon) =\int_{\mathbb{R}^d} (e^{\frac{u_\varepsilon \left( \left||x_\varepsilon|^{\frac{1}{\varepsilon}-1}x_\varepsilon+h \right|^{\varepsilon-1}(|x_\varepsilon|^{\frac{1}{\varepsilon}-1}x_\varepsilon+h), t_\varepsilon \right)-u_\varepsilon(x_\varepsilon,t_\varepsilon)}{\varepsilon}}-1) \frac{\beta_\varepsilon(x_\varepsilon, \frac{h}{|h|})dh}{|h|^{d+2\alpha}}.
\]
Note that, from the definition of $(x_\varepsilon,t_\varepsilon)$, we have for all $y \in \mathbb{R}^d$ :
\[(u_\varepsilon-\underline{\varphi}_\varepsilon- u_{+,\varepsilon})(x_\varepsilon,t_\varepsilon)\leq (u_\varepsilon-\underline{\varphi}_\varepsilon- u_{+,\varepsilon})(y,t_\varepsilon) ,\]
and thus by \eqref{HC}
\[-L^{\alpha}_{\varepsilon}(e^{\frac{\underline{\varphi}_\varepsilon}{\varepsilon}}n_{+,\varepsilon})(e^{\frac{\underline{\varphi}_\varepsilon}{\varepsilon}}n_{+,\varepsilon})^{-1}(x_\varepsilon,t_\varepsilon)\leq -L^{\alpha}_\varepsilon (n_\varepsilon) n_\varepsilon^{-1}(x_\varepsilon, t_\varepsilon) .\]
Finally, we obtain
\begin{equation}\label{inegalité_KPP_step1}
\begin{aligned}
(- L^{\alpha}_\varepsilon (n_\varepsilon) n_\varepsilon^{-1} + \dfrac{F_\varepsilon(x,n_\varepsilon)}{n_\varepsilon})(x_\varepsilon,t_\varepsilon) &\geq (-L^{\alpha}_{\varepsilon}(e^{\frac{\underline{\varphi}_\varepsilon}{\varepsilon}}n_{+,\varepsilon}))(e^{\frac{\underline{\varphi}_\varepsilon}{\varepsilon}}n_{+,\varepsilon})^{-1}+\dfrac{F_\varepsilon(x,n_\varepsilon)}{n_\varepsilon})(x_\varepsilon,t_\varepsilon)\\
&\geq  (-L^{\alpha}_{\varepsilon}(e^{\frac{\underline{\varphi}_\varepsilon}{\varepsilon}})(e^{-\frac{\underline{\varphi}_\varepsilon}{\varepsilon}})-\dfrac{F_\varepsilon(x,n_{+,\varepsilon})}{n_{+,\varepsilon}}\\
&+\widetilde{K}_\varepsilon (e^{\frac{\underline{\varphi}_\varepsilon}{\varepsilon}},n_{+,\varepsilon})(e^{\frac{\underline{\varphi}_\varepsilon}{\varepsilon}}n_{+,\varepsilon})^{-1}+\dfrac{F_\varepsilon(x,n_\varepsilon)}{n_\varepsilon})(x_\varepsilon,t_\varepsilon)\\
&\geq (o(1)+\dfrac{F_\varepsilon(x,n_\varepsilon)}{n_\varepsilon}-\dfrac{F_\varepsilon(x,n_{+,\varepsilon})}{n_{+,\varepsilon}})(x_\varepsilon,t_\varepsilon).
\end{aligned}
\end{equation}
We have to note that thanks to Lemma \ref{lemme_utile}, in the last inequality, we have controlled the terms : \[o(1) \leq \widetilde{K}_\varepsilon (e^{\frac{\underline{\varphi}_\varepsilon}{\varepsilon}},n_{+,\varepsilon})(e^{\frac{\underline{\varphi}_\varepsilon}{\varepsilon}}n_{+,\varepsilon})^{-1}(x_\varepsilon,t_\varepsilon)\ \text{ and } \  o(1) \leq -L^{\alpha}_{\varepsilon}(e^{\frac{\underline{\varphi}_\varepsilon}{\varepsilon}})(e^{-\frac{\underline{\varphi}_\varepsilon}{\varepsilon}})(x_\varepsilon,t_\varepsilon).\]
Finally combining \eqref{deérivée_temporelle_KPP_Step1} and \eqref{inegalité_KPP_step1}, we obtain
\begin{equation}\label{inegalité2_KPP_step1}
o(1) \leq \dfrac{F_\varepsilon(x_\varepsilon, n_{+,\varepsilon})}{n_{+,\varepsilon}}(x_\varepsilon, t_\varepsilon)-\dfrac{F_\varepsilon(x_\varepsilon, n_{\varepsilon})}{n_{\varepsilon}}(x_\varepsilon, t_\varepsilon).
\end{equation}

We are going to prove by contradiction that \eqref{inegalité2_KPP_step1} implies $o(1)+n_{+,\varepsilon}(x_\varepsilon) \leq n_\varepsilon (x_\varepsilon, t_\varepsilon)$. Let's suppose that there exists a subsequence $(\varepsilon_k)_{k\in \mathbb{N}}$ and a positive constant $C$ such that 
\[n_{\varepsilon_k}(x_{\varepsilon_k},t_{\varepsilon_k})+C<n_{+,\varepsilon_k}(x_{\varepsilon_k}).\]
Then thanks to the strict monotony of the function $s\mapsto \frac{F(x,s)}{s}$ (assumption $(iv)$ in \eqref{H4}) and the mean value Theorem there exists a sequence $y_k$ such that 
\begin{align*}
o_{\varepsilon_k}(1) &\leq \dfrac{F(x_{\varepsilon_k}, n_{+,\varepsilon_k})}{n_{+,\varepsilon_k}}(x_{\varepsilon_k}, t_{\varepsilon_k})-\dfrac{F(x_{\varepsilon_k}, n_{\varepsilon_k})}{n_{\varepsilon_k}}(x_{\varepsilon_k}, t_{\varepsilon_k})\\
&=\partial_s(\frac{F(x_{\varepsilon_k}, s)}{s})(y_{\varepsilon_k})(n_{+,\varepsilon_k}-n_{\varepsilon_k})(x_{\varepsilon_k},t_{\varepsilon_k})\\
&\leq - \overline{C} \times C.
\end{align*}
This is a contradiction. Therefore, for $\varepsilon$ small enough, 
\[n_{+,\varepsilon}(x_\varepsilon)+o(1) \leq n_\varepsilon (x_\varepsilon,t_\varepsilon)  \Rightarrow 1+o(1) \leq \dfrac{n_\varepsilon(x_\varepsilon,t_\varepsilon)}{n_{+,\varepsilon}(x_\varepsilon)}.\]

To bring back this inequality at the point $(x_0,t_0)$, we use exactly the same arguments as for the proof of Theorem \ref{theorem_principal_2} by considering a disjunction of cases $|x_\varepsilon|<|x_0|$ and $|x_0|<|x_\varepsilon|$. We do not provide the details of this disjunctions of cases since they are the same. \\
So we have proved, in all cases
\[1 +o(1) \leq \dfrac{n_\varepsilon (x_0,t_0)}{n_{+,\varepsilon}(x_0)}.\]

The second step can also be proved following similar arguments as in the previous step, thus we do not provide the demonstration.
\end{proof}

\appendix
\section{The proof of Lemma \ref{lemme_utile}}
All along the appendix, we will denote by $C$ positive constants that can change from line to line. 
\begin{proof}[\textbf{Proof of (i).}] 
We are going to follow the appendix A of \cite{Mirrahimi1}. 

Let $\delta < \frac{1}{2}$ be a positive constant. By a compactness argument, we only have to prove it for $|x|>1$. We compute 

\begin{align*}
\left|\frac{L^{\alpha}(g)(x)}{g(x)}\right|&=\left|\int_{\mathbb{R}^d}\left(\frac{1+|x|^{d+2\alpha}}{1+|x+h|^{d+2\alpha}}-1\right)\frac{\beta (x,\frac{h}{|h|})dh}{|h|^{d+2\alpha}}\right| \\
&\leq \int_{\mathbb{R}^{d}\backslash[ B(-x,\delta|x|) \cup B(0,\delta)]}\left|\frac{1+|x|^{d+2\alpha}}{1+|x+h|^{d+2\alpha}}-1\right|\frac{\beta (x,\frac{h}{|h|}) dh}{|h|^{d+2\alpha}} \\
&+\int_{ B(-x,\delta|x|) \backslash B(0,\delta)}\left|\frac{1+|x|^{d+2\alpha}}{1+|x+h|^{d+2\alpha}}-1\right|\frac{\beta (x,\frac{h}{|h|}) dh}{|h|^{d+2\alpha}}\\
&+\int_{B(0,\delta)}\left|\frac{1+|x|^{d+2\alpha}}{1+|x+h|^{d+2\alpha}}-1\right|\frac{\beta (x,\frac{h}{|h|})dh}{|h|^{d+2\alpha}}\\
&=I_1 + I_2 + I_3.
\end{align*}
Let us begin by approximating $I_1$. 
\begin{align*}
I_1 &= \int_{\mathbb{R}^{d}\backslash[ B(-x,\delta|x|) \cup B(0,\delta)]}\left|\frac{1+|x|^{d+2\alpha}}{1+|x+h|^{d+2\alpha}}-1\right| \frac{\beta (x,\frac{h}{|h|})dh}{|h|^{d+2\alpha}} \\
& \leq \int_{\mathbb{R}^{d}\backslash[ B(-x,\delta|x|) \cup B(0,\delta)]}\left| \frac{C}{\delta^{d+2\alpha}}-1\right|\frac{\beta (x,\frac{h}{|h|}) dh}{|h|^{d+2\alpha}} \\
&\leq \frac{(C+1)}{\delta^{d+2\alpha}} B \ \mathrm{mes}(S^{d-1})\int_{\delta}^{+\infty} \frac{dh}{|h|^{1+2\alpha}} =  \frac{C}{\delta^{d+4\alpha}}.  
\end{align*}
For $I_2$, we write:
\begin{align*}
I_2 &= \int_{ B(-x,\delta|x|) \backslash B(0,\delta)}\left|\frac{1+|x|^{d+2\alpha}}{1+|x+h|^{d+2\alpha}}-1\right|\frac{\beta (x,\frac{h}{|h|}) dh}{|h|^{d+2\alpha}}\\
& \leq B \int_{ B(-x,\delta|x|) \backslash B(0,\delta)}\left|\frac{1+|x|^{d+2\alpha}}{1+|x+h|^{d+2\alpha}}-1\right|\frac{dh}{|h|^{d+2\alpha}}\\
& \leq B \int_{ B(-x,\delta|x|) \backslash B(0,\delta)}\frac{\left||x|^{d+2\alpha}-|x+h|^{d+2\alpha}\right|}{1+|x+h|^{d+2\alpha}} \frac{dh}{|h|^{d+2\alpha}}\\
& \leq B \int_{ B(-x,\delta|x|) \backslash B(0,\delta)}\frac{|x|^{d+2\alpha}+|x+h|^{d+2\alpha}}{1+|x+h|^{d+2\alpha}} \frac{dh}{|h|^{d+2\alpha}}\\
& \leq B \int_{ B(-x,\delta|x|) \backslash B(0,\delta)}\frac{|x|^{d+2\alpha}+|\delta x|^{d+2\alpha}}{1+|x+h|^{d+2\alpha}} \frac{dh}{|h|^{d+2\alpha}}\\
& \leq C \int_{ B(-x,\delta|x|) \backslash B(0,\delta)}\frac{1}{1+|x+h|^{d+2\alpha}} \frac{|x|^{d+2\alpha}}{|h|^{d+2\alpha}}dh.
\end{align*}
But we know that $h \in B(-x,\delta|x|) \backslash B(0,\delta)$, using that $\delta<\frac{1}{2}<|x|$, we deduce that 
\[|x|(1-\delta) \leq |h| \leq (1+\delta)|x|\Rightarrow \left| \frac{x}{h} \right| \leq \left| \frac{1}{1-\delta} \right|.\]
Thus, we deduce
\begin{align*}
I_2& \leq \frac{C}{(1-\delta)^{d+2\alpha}}  \int_{ B(-x,\delta|x|) \backslash B(0,\delta)}\frac{1}{1+|x+h|^{d+2\alpha}}dh \\
& \leq \frac{C}{(1-\delta)^{d+2\alpha}} \int_{0}^{\delta|x|}\frac{r^{d-1}}{1+r^{d+2\alpha}}dr\\
& \leq \frac{C}{(1-\delta)^{d+2\alpha}} \int_{0}^{\infty}\frac{r^{d-1}}{1+r^{d+2\alpha}}dr.
\end{align*}

%
%
To control $I_3$, we write $I_3$ in the following form:
\[I_3=C\left|\int_{B(0,\delta)}\left(\frac{1+|x|^{d+2\alpha}}{1+|x+h|^{d+2\alpha}}+\frac{1+|x|^{d+2\alpha}}{1+|x-h|^{d+2\alpha}}-2\right)\frac{\beta (x,\frac{h}{|h|})dh}{|h|^{d+2\alpha}}\right|.\]
Next, we define
\[f(x,h):=\frac{1+|x|^{d+2\alpha}}{1+|x+h|^{d+2\alpha}}.\]
Since for all $x \in \mathbb{R}^d$, the map that $(h \mapsto f(x,h))$ is $C^{1+2\alpha}$, we know that $I_3$ is well defined. Moreover for every $h \in B(0,\delta) \backslash \left\lbrace 0 \right\rbrace$, when the parameter $|x|$ tends to $\infty$, we have that $\frac{(f(x,h)+f(x,-h)-2)\beta(x,\frac{h}{|h|})}{|h|^{d+2\alpha}}$ tends to $0$. So we deduce thanks to the dominated convergence Theorem, that $(x \mapsto I_3(x))$ tends to $0$ when $|x|$ tends to $\infty$. According to the continuity of the maps $(x \mapsto f(x,h))$ and $(x \mapsto \beta(x, \theta))$, we deduce that the map $(x \mapsto I_3(x))$ is continuous and so  we conclude that $I_3$ is bounded independently of $x$. We refer to \cite{Mirrahimi1} for more details (see the Annex A1).\\
Combining the above inequalities, we obtain that there exists a constant $C$ such that for all $x \in \mathbb{R}^d$, 
\begin{equation}
| L^{\alpha} g (x) | \leq C g(x). 
\end{equation}
Using the above inequality, we can conclude with a change of variable $z=ay$:
\begin{align*}
| L^{\alpha} g (ax) |& =| \int_{\mathbb{R}^d} \dfrac{g(ax)-g(ax+ay)}{|y|^{d+2\alpha}}\beta(ax,\frac{y}{|y|})dy|\\
&=|\int_{\mathbb{R}^d} \dfrac{g(ax)-g(ax+z)}{|a^{-1}z|^{d+2\alpha}}a^{-d} \beta(ax,\frac{a^{-1}z}{|a^{-1}z|})dz|\\
&=a^{2\alpha} |L^\alpha (g)(ax)|\\
&\leq C a^{2\alpha} g(ax).
\end{align*}
Finally, we obtain 
\[|L^\alpha g(ax) |\leq C a^{2\alpha} g(ax).\]

\textbf{\textit{Proof of (ii).}} Since all the functions involved in $\widetilde{K}$ are differentiable, and thanks to the dominated convergence theorem, we deduce that $\widetilde{K}$ is continuous. We can note the following fact:
\begin{equation}
|\nabla g(x)| = O(|x|^{-(d+2\alpha+1)}) \ \text{ as } |x|\rightarrow +\infty.
\label{eq6}
\end{equation}
With the change of variable $\widetilde{x}= a y$, we find:
\begin{equation}\label{eq7}
\widetilde{K}(g (a . ),\chi )(x)= a^{2\alpha} C' \  PV \int_{\mathbb{R}^d}\frac{(g(a x)-g(\widetilde{x}))(\chi( x)-\chi(a^{-1} \widetilde{x}))}{|a x-\widetilde{x}|^{d+2\alpha}}\beta (x,\frac{ax-\widetilde{x}}{|ax-\widetilde{x}|})d\widetilde{x}.
\end{equation}
Since $\chi \in C^1(\mathbb{R}^d)\cap L^{\infty}(\mathbb{R}^d)$, $\beta \in L^{\infty}(\mathbb{R}^d \times S^{d-1})$ and $g \in C^1(\mathbb{R}^d)\cap L^{\infty}(\mathbb{R}^d)$ this integral converges in $\mathbb{R}^d$. For $x\in \mathbb{R}^d$, we have to estimate 
\[J(x)=a^{2\alpha} C' \ PV\int_{\mathbb{R}^d}\frac{(g(x)-g(\widetilde{x}))(\chi( a^{-1}x)-\chi(a^{-1} \widetilde{x}))}{| x-\widetilde{x}|^{d+2\alpha}}\beta (a^{-1}x,\frac{x-\widetilde{x}}{|x-\widetilde{x}|})d\widetilde{x}\]
at point $a x$. We define for $x\in \mathbb{R}^d$
\[J_1(x)=a^{2\alpha} C'\ PV\int_{B(x,1)}\frac{(g(x)-g(\widetilde{x}))(\chi( a^{-1}x)-\chi(a^{-1} \widetilde{x}))}{| x-\widetilde{x}|^{d+2\alpha}}\beta (a^{-1}x,\frac{x-\widetilde{x}}{|x-\widetilde{x}|})d\widetilde{x},\]
\[\text{and } J_2(x)=a^{2\alpha} C' \ PV\int_{\mathbb{R}^d\backslash B(x,1)}\frac{(g(x)-g(\widetilde{x}))(\chi( a^{-1}x)-\chi(a^{-1} \widetilde{x}))}{| x-\widetilde{x}|^{d+2\alpha}}\beta (a^{-1}x,\frac{x-\widetilde{x}}{|x-\widetilde{x}|})d\widetilde{x},\]
so that $J=J_1+J_2$.
We split the proof in two parts: when $|x|\leq M$ and when $|x|>M$ with $M$ a positive constant arbitraly large.\\
For $|x|\leq M$, according to \eqref{eq7}
\[\widetilde{K}(g(a \cdot), \chi) (x) = a^{2\alpha} J(ax).\]
First we prove the existence of a constant $C$ large enough such that 
\begin{equation}\label{eq8}
\forall x \in \overline{B}(0,M), \ |J(x)|\leq Cg(x).
\end{equation}
Since $|J|$ is continuous, we deduce that in $\overline{B}(0,M)$, $|J|$ is bounded by a constant $D$. Thus, since $g$ si decreasing, if we take $C$ larger than $D \times (1+M^{d+2\alpha})$, the assertions \eqref{eq8} holds true. Since $a<1$, we conclude that for all $x \in \overline{B}(0,M)$:
\[ |\widetilde{K}(g(a\cdot),\chi)(x)| = a^{2\alpha}|J(ax) | \leq a^{2\alpha} \dfrac{C}{1+|ax|^{d+2\alpha}} \leq a^{2\alpha-\gamma} C \dfrac{1+|x|^{d+2\alpha}}{1+|ax|^{d+2\alpha}}  \dfrac{1}{1+|x|^{d+2\alpha}} \leq a^{2\alpha-\gamma} Cg(x) .\]

For $|x|>M$, we first study $J_1$ and then $J_2$.\\
\textit{Estimate of $J_1$:} From the formula \eqref{eq6}, for $|x|>M$, since $\chi$ is $C^{1}(\mathbb{R}^d)$ and periodic, and since $\gamma < 1$ and $2\alpha -\gamma$ is strictly positive, we have:
\begin{align*}
|J_1(x)| &\leq CB \int_{B(x,1)} \frac{a^{2\alpha-\gamma}|x-\widetilde{x}|^{\gamma}}{|x-\widetilde{x}|^{d+2\alpha}}\underset{z\in[x;\widetilde{x}]}{\sup}|\nabla g(z)||x-\widetilde{x}|d\widetilde{x} \\
&\leq C \frac{a^{2\alpha - \gamma}}{|x|^{d+2\alpha}}\int_{B(x,1)}\frac{1}{|x-\widetilde{x}|^{d+2\alpha-\gamma-1}}d\widetilde{x} \\
&\leq \frac{a^{2\alpha - \gamma}D_1}{1+|x|^{d+2\alpha}}.
\end{align*}

\textit{Estimate of $J_2 $:} Since $\chi$ is bounded and $a < 1$, we obtain:
\begin{align*}
|J_2(x)| &\leq a^{2\alpha} CB \int_{|y|\geq 1} \frac{g(x)}{|y|^{d+2\alpha}}dy +a^{2\alpha} CB \int_{|y|\geq 1} \frac{g(x+y)}{|y|^{d+2\alpha}}dy \\
&\leq a^{2\alpha}CB g(x)+a^{2\alpha}CB \int_{|y|\geq \frac{|x|}{2}} \frac{g(x+y)}{|y|^{d+2\alpha}}dy+a^{2\alpha}CB \int_{1\leq|y|\leq\frac{|x|}{2}} \frac{g(x+y)}{|y|^{d+2\alpha}}dy \\
&\leq a^{2\alpha}\frac{CB}{|x|^{d+2\alpha}}+a^{2\alpha} \frac{2^{d+2\alpha}CB}{|x|^{d+2\alpha}}\int_{\mathbb{R}^d} g(y)dy+a^{2\alpha}CB \int_{1\leq|y|\leq\frac{|x|}{2}} \frac{g(\frac{x}{2})}{|y|^{d+2\alpha}}dy \\
&\leq a^{2\alpha}\frac{CB}{|x|^{d+2\alpha}}+a^{2\alpha} \frac{2^{d+2\alpha}CB}{|x|^{d+2\alpha}}\int_{\mathbb{R}^d} g(y)dy+a^{2\alpha}CB 2^{d+2\alpha}g(x) \int_{|y|\geq 1} \frac{1}{|y|^{d+2\alpha}}dy \\
&\leq a^{2\alpha-\gamma}\frac{D_2}{1+|x|^{d+2\alpha}}.
\end{align*}
The third line is obtained noting that for $|y|\leq \frac{|x|}{2}$, we have $\frac{|x|}{2} \leq |x|-|y| \leq |x+y|$. 

Putting all together we find the existence of $C$ such that \textit{(ii)} holds.
\end{proof}

\textbf{Acknowledgments}

I want to thank Sepideh Mirrahimi and Jean-Michel Roquejoffre for all their precious advice and all the fruitful discussions. {The  research  leading to  these  results  has  received  fundings  from  the  European  Research  Council  under the European Union's Seventh Framework Program (FP/2007-2013) / ERC Grant Agreement n.321186 - ReaDi - Reaction-Diffusion Equations, Propagation and Modeling and from the french ANR project MODEVOL ANR-13-JS01-0009}. 
\bibliographystyle{plain} 

\bibliography{/home/aleculie/Dropbox/Documents/Bibliographie/bibliographie.bib} 


\end{document}